\documentclass[11pt]{article}

\usepackage{graphicx}
\usepackage{amssymb}
\usepackage{amsmath}
\usepackage{amsthm}
\usepackage{latexsym} 
\usepackage{color}
\usepackage{pstricks} 
\usepackage{psfrag}
\usepackage{rotating}

\newtheorem{theorem}{Theorem}[section]

\newtheorem{lemma}[theorem]{Lemma}
\newtheorem{corollary}[theorem]{Corollary}
\newtheorem{cor}[theorem]{Corollary}

\newtheorem{proposition}[theorem]{Proposition}

\newtheorem{innercustomthm}{Theorem}

\newenvironment{customthm}[1]
  {\renewcommand\theinnercustomthm{#1}\innercustomthm}
  {\endinnercustomthm}
\newtheorem*{thmone}{Theorem~\ref{thm:one}}

\newtheorem*{thmtwo}{Theorem~\ref{thm:two}}
\newtheorem*{thmincidence}{Theorem~\ref{thm:incidence}}
\newtheorem*{thminfty}{Theorem~\ref{thm:infty}}

\theoremstyle{definition}
\newtheorem{definition}[theorem]{Definition}

\newtheorem{remark}[theorem]{Remark}

\usepackage{hyperref}
\usepackage{comment}

\newcommand\CCC{\mathcal{C}}

\newcommand\PP{\mathcal{P}}
\newcommand\R{\mathbb{R}}

\newcommand\Z{\mathbb{Z}}

\newcommand\Isom{\text{Isom}}
\newcommand\Stab{\text{Stab}}

\renewcommand\hat{\widehat}

\newcommand{\cD}{\mathcal{D}}

\usepackage{authblk}
\title{On groups with Schottky set boundary} 
\author[1]{Peter Ha\"issinsky}
\author[2]{Luisa Paoluzzi}
\author[3]{Genevieve S. Walsh}
\affil[1]{Aix-Marseille Univ, CNRS, I2M, Marseille, France,
phaissin@math.cnrs.fr}

\affil[2]{Aix-Marseille Univ, CNRS, I2M, Marseille, France, luisa.paoluzzi@univ-amu.fr  }

\affil[3]{(Corresponding author) Tufts University Dept. of Math, Medford, USA,
genevieve.walsh@tufts.edu}

\date{}
\begin{document} 
\maketitle 

\abstract{ We study relatively hyperbolic group pairs whose boundaries are Schottky sets. We characterize the groups that have boundaries where the Schottky sets have incidence graphs with 1 or 2 components.} 

\vskip 2mm


\vskip 2mm


\section{Introduction} 
Convergence group actions on the 2-sphere were introduced by Gehring and Martin in \cite{GM1} and have been studied extensively since then.  It was conjectured 
by Martin and Skora in \cite{martin:skora} that every faithful convergence group action $G$ on $S^2$ by orientation preserving homeomorphisms is \emph{covered} by the induced action of a discrete subgroup of $\Isom(\mathbb{H}^3)$ on $S^2$, i.e., there exist a Kleinian group $K$, an isomorphism $\rho : K\to G$ and a degree 1 map $\phi:\hat{\mathbb{C}} \to S^2$ such that
the following diagram commutes:
$$\begin{array}{ccccc}
K & \curvearrowright & \hat{\mathbb{C}}  & \longrightarrow &\hat{\mathbb{C}} \\
& & & & \\
\Big\downarrow \rho  & & \Big\downarrow \phi & & \Big\downarrow \phi\\
& & & & \\
G &\curvearrowright  & S^2 & \longrightarrow & S^2
\end{array}$$
This remains open. This conjecture is closely related to Cannon's conjecture \cite{Cannon91}, which asserts that a 
hyperbolic group with 2-sphere boundary is virtually a discrete subgroup of $\Isom(\mathbb{H}^3)$.
Indeed, if a hyperbolic group $G$ acts faithfully on its Gromov boundary $\partial G$ and $\partial G = S^2$, then Martin and Skora's conjecture posits that its finite index  subgroup of
orientation preserving elements is conjugate to a Kleinian group.  

Here we are dealing with the case of relatively hyperbolic groups, and specifically those whose boundaries are {\it topological Schottky sets}. These are defined and described in Section \ref{sec:topschott}. 
Some familiar examples are the Sierpi\'nski carpet and the Apollonian gasket. 
Our motivation for studying those groups is essentially twofold. Firstly, there are many examples of groups that admit a peripheral structure for which their boundary is a topological Schottky set, cf. Theorem \ref{thm:infty}.
Secondly, the relatively hyperbolic groups with these boundaries are all conjectured to be virtually discrete subgroups of $\Isom(\mathbb{H}^3)$, see \cite{hruska:walsh}. Here we show that many relatively hyperbolic groups with boundaries that are topological Schottky sets are virtually discrete subgroups of $\Isom(\mathbb{H}^3)$. Furthermore, we say which Kleinian groups arise when the boundaries are certain types of topological Schottky sets. 

We say that $(G,\mathcal{P})$ is a {\it relatively hyperbolic group pair} if $(G, \mathcal{P})$ acts as a geometrically finite convergence group on the compactification $X\cup \partial X$ of a hyperbolic space $X$. See section \ref{sec:general} for the detailed definition.  In this case, we say that the Gromov boundary of $X$, $\partial X$, is the Bowditch boundary of $(G, \mathcal{P})$, denoted $\partial(G, \mathcal{P}).$  We also call $\partial(G, \mathcal{P})$ the {\it relatively hyperbolic boundary} or sometimes just ``the boundary".  Throughout, $(G, \mathcal{P})$ is a non-elementary relatively hyperbolic group pair (besides Proposition \ref{prop:elementary} where the classification of elementary convergence groups acting on $S^2$ is provided), which means that $\partial(G, \mathcal{P})$ has more than two points.  

Following \cite{BKM}, we define a \emph{Schottky set} as the complement of at least three  
disjoint open round balls in the $n$-sphere $S^n$, where $S^n$ is equipped with the standard metric as a subset of Euclidean space.  Throughout this paper, we will restrict ourselves to $n=2$, so all our Schottky sets are planar. We actually work with the non-metric analog, which we call {\it topological} Schottky sets. 

Due to the properties of a topological Schottky set $\mathcal{S}$ in Definition \ref{def:schottky}, every $\mathcal{S}$ produces an {\it incidence graph $\Gamma(\mathcal{S})$}, the simplicial graph whose vertices correspond to the open disks $\{D_i\}_{i\in I}$ of its complement in $S^2$, and whose edges correspond to (1-point) incidences between closures of the disks $D_i$.

Our main results are as follows: 

\begin{customthm}{A} 
\label{thm:incidence} 
Let $\mathcal{S}$ be a topological Schottky set with $\mathcal{S} = \partial(G, \mathcal{P})$.   Then the incidence graph $\Gamma(\mathcal{S})$ has 1, 2 or infinitely many components. 
The stabilizer of each component is virtually the deck transformation group of a regular planar covering of a closed surface. 
\end{customthm}

\begin{customthm}{B}
    
 \label{thm:one} Let $\mathcal{S}$ be a topological Schottky set with $\mathcal{S} = \partial(G, \mathcal{P})$.
 
When the incidence graph $\Gamma(\mathcal{S})$ has one component, then $G$ contains a finite index subgroup $H$ that is isomorphic to
a free product of a free group $F_n$ of rank $n\geq 0$ and some finite index subgroups of groups in $\mathcal{P}$. Moreover, the action of the group $H$ on $\mathcal{S}$ is faithful and orientation preserving, and $H$ is covered by a geometrically finite Kleinian group $K$.
\end{customthm} 
\begin{remark} The Apollonian gasket satisfies the hypotheses of both theorems.  In this case, $G$ is always a free group of rank 2 \cite{HPW}.  This is also the covering group of an infinite planar surface covering a genus-2 surface. \end{remark}

From a topological viewpoint, $K$ contains a finite-index torsion-free subgroup that uniformizes a 3-manifold 
obtained by gluing together along compression disks a handlebody and $I$-bundles over surfaces.

 \begin{customthm}{C} \label{thm:two} Let $\mathcal{S}$ be a topological Schottky set with $\mathcal{S} = \partial(G, \mathcal{P})$.
When the incidence graph $\Gamma(\mathcal{S})$  has exactly  2 components, $G$ is virtually a closed surface group. \end{customthm}

In contrast, when the incidence graph has infinitely many components, then the group  is covered by a geometrically finite convergence group that may have a Sierpi\'nski carpet boundary. Showing that these are essentially Kleinian is still a wide open question, even in the word hyperbolic case, cf. \cite{KK}. Note that Theorem \ref{thm:infty} below enables us to construct examples of Schottky limit sets that have infinitely many components of their incidence graphs but that do not come from a Sierpi\'nski carpet limit set. For example, apply the theorem to a geometrically finite Kleinian group that contains a rank-2 accidental parabolic fixed point (see for instance the first example in \cite{Brock}, illustrated by Figure 6 therein).
So far, all the examples we know of with Sierp\'inski carpet boundary are virtually fundamental groups of hyperbolic 3-manifolds with totally geodesic boundary (which may have cusps), and this is consistent with conjectures in \cite{KK} and \cite{hruska:walsh}. 

 \begin{customthm}{D}
     \label{thm:infty}
  Let $K$ be a geometrically finite Kleinian group with non-empty domain of discontinuity. Then there is a peripheral structure  $\mathcal{P}_{K'}$ on a finite index subgroup $K'$ of $K$, such that $(K', \mathcal{P}_{K'})$  is a relatively hyperbolic group pair and $\partial(K', \mathcal{P}_{K'})$ is a topological Schottky set. 
 Moreover, $\mathcal{P}_{K'}$ contains 
 the natural peripheral structure of the Kleinian group $K'\subset K$.
 \end{customthm}

In Section \ref{sec:general} we prove some general facts about relatively hyperbolic groups, generalizing some theorems about hyperbolic boundaries to relatively hyperbolic boundaries. In Section \ref{sec:S2} 
we restrict to relatively hyperbolic groups 
that are geometrically finite convergence groups acting on $S^2$. 
Although the results in this section will be used later in the context of topological Schottky sets, they do not only apply to this specific context.
In Section \ref{sec:blow} we describe how to ``blow up" 2-ended peripheral subgroups in geometrically finite groups acting on $S^2$.  This will change the peripheral structure, but not the group; moreover the group with its new peripheral structure is shown to admit again a convergence action on the 2-sphere. In Sections \ref{sec:topschott} and \ref{sec:incidence} we introduce and discuss topological Schottky sets and their incidence graphs, and prove Theorem \ref{thm:incidence}. Finally in Section \ref{sec:one} we prove Theorem \ref{thm:one}, and in Section \ref{sec:more} we prove Theorems \ref{thm:two} and \ref{thm:infty}. 

\subsection*{Acknowledgements} We thank the CIRM (Centre International de Rencontres Math\'ematiques) in Luminy, Marseille,  where this work began. The first author was partially supported by ANR-22-CE40-0004 GoFR. The third author was partially supported by the NSF, through DMS 2005353 and 2405033. We thank the anonymous referees for constructive comments that helped the authors improve this text.

\section{Relative hyperbolicity and relative quasiconvexity}
\label{sec:general} 
Here we provide some results about general relatively hyperbolic groups and their boundaries. References on metric spaces in the sense of Gromov include \cite{ghys:delaharpe:groupes, bridson:haefliger}.
Let $G$ be a finitely generated  group
and a family $\mathcal{P}$ of subgroups consisting of finitely many conjugacy classes. 

Let us first recall that a \emph{convergence group} $G$ is a group of homeomorphisms of a compact metric space $Z$ such that any sequence $(g_n)_n$ of distinct elements contains a convergent subsequence, i.e., up to a subsequence, there are  two points $a$ and $b$ in $Z$ so that $(g_n)$ tends uniformly to the constant map $a$ on compact subsets of $Z\setminus\{b\}$.
One may then define the limit set $\Lambda_G$ as the set of limit points $a$ of all convergence 
sequences in $G$. It is a compact invariant subset of $Z$. Its complement, $\Omega_G$, is the ordinary set: the action of $G$ on $\Omega_G$ is properly discontinuous, see \cite{GM1} for more properties. 
Note that any discrete group of isometries on a geodesic, proper, hyperbolic space $X$ admits a convergence action on $X\cup \partial X$.

\begin{definition}[\cite{bowrelhyp}] \label{def:original rel hyp}  The pair $(G,\mathcal{P})$ is \textit{relatively hyperbolic} if $G$ acts on $X$ properly discontinuously and by isometries, where $X$ is a proper hyperbolic geodesic metric space such that:

\begin{enumerate}
	\item each point of $\partial X$ is either a \emph{conical limit point} or a \emph{bounded parabolic point}. 
	\item $\mathcal{P}$ is exactly the collection of maximal \emph{parabolic subgroups}.
\end{enumerate}

A conical limit point is a point $y \in \partial X$ such that
there exists a sequence $(g_i)$ in $G$  and distinct points $a, b \in \partial X$, such that $g_i(y) \rightarrow a$ and $g_i(z) \rightarrow b$, for all $z \in \partial X \setminus \{y\}$. A parabolic point $y_P$ is a point with an infinite stabilizer that fixes no other point, i.e., the fixed point of a parabolic subgroup $P$. It is  bounded if $(\partial X \setminus \{y_P\})/P$ is compact. 
Whenever we have a properly discontinuous action by isometries and these two conditions are satisfied, we say $(G, \mathcal{P})$ acts \textit{geometrically finitely} on $X$.  If $(G, \mathcal{P})$ is a relatively hyperbolic pair, then $\partial( G, \PP)\ = \partial X $ is its \textit{Bowditch boundary}, or \textit{relatively hyperbolic boundary}.  This depends on $\mathcal{P}$, but is well-defined for the pair $(G, \mathcal{P})$. 

\end{definition} 

There are many equivalent definitions of relatively hyperbolic groups, see \cite{Hruskasurvey}.  In particular, a relatively hyperbolic group admits a {\it cusp-uniform} action on a hyperbolic metric space $X$.  By definition, this action is cocompact on the complement of a $G$-invariant collection of horoballs in $X$, which are centered at points of the boundary of the hyperbolic metric space $X$.  See \cite[Chapter 3]{BuyaloS}.

As we will be using topological properties of Bowditch boundaries, we recall two topological notions that will be used several times.

\begin{definition}[Null sequences and $E$-sets]\label{def:nses}
Given a compact metric space $Z$, a \emph{null-sequence} is a collection of subsets $\mathcal{C}$ such that, for any $\delta >0$, the collection $\mathcal{C}$ contains at most finitely many elements of diameter at least $\delta$.

An \emph{$E$-set} is a connected compact subset of the sphere $S^2$ such that the collection of connected components of its complement is a null-sequence.
\end{definition}

 Next we recall the definition of a relatively quasi-convex subgroup of a relatively hyperbolic group pair, which is a generalization of the notion of a quasi-convex subgroup of a hyperbolic group.  This definition is from \cite{Hruskasurvey}, where its equivalence to several other definitions is proven. 
\begin{definition} Let $H$ be a subgroup of a group $G$ such that $(G, \mathcal{P})$ is a relatively hyperbolic group pair.  We say that $H$ is relatively quasiconvex if every point of the limit set of $H$ is either a conical limit point or a bounded parabolic point for the action of $H$. That is, let $H \cap \mathcal{P} = \lbrace H \cap P: P \in \mathcal{P}, |H\cap P| = \infty \rbrace$. Then $(H, H \cap \mathcal{P})$ is a relatively hyperbolic group pair. 
\end{definition}

 We emphasize that there are many different definitions. In particular, subgroups which are relatively quasiconvex are {\it dynamically quasiconvex} in the sense of \cite[Definition 4.9]{GPdyn}, which implies point (1) in Proposition \ref{prop:nullseq} below.  This characterizes a subgroup as relatively quasiconvex by looking at the orbit of the limit set of the subgroup.  We provide an elementary proof for the convenience of the reader. 

\begin{proposition}\label{prop:nullseq} Let $(G,\PP)$ be relatively hyperbolic.
\begin{enumerate}
\item If $K$ is the limit set of a relatively quasiconvex subgroup, then the set of elements in the orbit $GK$ forms a null-sequence. 

\item Let $\CCC$ be a $G$-invariant collection of compact subsets of $\partial( G,\PP)$ which defines a null-sequence, where each element of $\CCC$ contains more than one point. 
Then $\CCC/G$ is finite and, for any  perfect set $K\in\CCC$, $\Stab (K)$ is a relatively quasiconvex subgroup with
limit set $K$. 
\end{enumerate}
\end{proposition}

\begin{proof}
Let us first consider a 
geometrically finite action of the group $G$ on a proper geodesic hyperbolic metric space $X$ so that
the stabilizers of the parabolic points are the elements of $\PP$. We may then identify $\partial X$ with $\partial(G,\PP)$ and endow
it with a visual distance seen from a base point $o\in X$. 

Let $H$ be a relatively quasiconvex subgroup of the relatively hyperbolic group $(G, \mathcal{P})$. We will prove that the orbit of its limit set $\Lambda_H = K$ forms a null sequence.  See \cite[Corollary 2.5]{GMRS} for the hyperbolic case. 

Fix $\delta >0$ and
let $R>0$ denote the upper bound on the distances from the origin $o$ to  any geodesic joining points $\delta$-apart in the boundary. 
Let us pick a  $G$-invariant collection of horoballs $\mathcal{H}$  in $X$ centered at the set of parabolic points 
in such a way  that they are   
pairwise disjoint and that their distance
to $o$ is at least $R+1$ (by shrinking if necessary). By abuse of notation, we will also let $\mathcal{H}$ denote the union of the horoballs of the collection.
Let $\CCC$ denote the set of translates $g(K)$ of diameter at least $\delta$ and assume that $K\in\CCC$.  Since $H$ is relatively quasiconvex,
there is some $q>0$ so that, for any geodesic $\gamma$  joining two points in $K$, $\gamma\cap (X\setminus \mathcal{H})$ is contained in the $q$-neighborhood of $Ho$, \cite[Definition 6.6]{Hruskasurvey}.

If $L=g(K)\in\CCC$, then we may find a geodesic $\gamma$ joining two points in $K$ such that $g(\gamma)$ is at distance at most $R$ from $o$. Since the horoballs are $G$-invariant and at distance at least $R+1$ from $o$, by the previous observation about $\gamma$, we may find a point of $g(\gamma)$ at distance at most $R$ from the origin and at distance at most $q$ from $gHo$. Thus, there exists $g_L\in gH$ such that $g_L(o)\in B(o,R+q)$.
 Since the action
of $G$ is properly discontinuous, there are finitely many elements $g\in G$ with $g(o)\in B(o,R+q)$, hence finitely many $L\in\CCC$. This shows  that $GK$ is a null-sequence.

\medskip

 We now establish point 2. 
 Let $m>0$ be such that any distinct pair of points of $\partial(G, \PP)$ can be $m$-separated by an element of $G$, i.e., for any $x,y\in\partial(G,\PP)$, $x\neq y$,
 there is some $g\in G$ such that $d(g(x),g(y))\ge m$.  Such $m$ exists since the action on the set of distinct pairs is cocompact, see \cite{TukiaConical}. 
Given $\delta>0$, we let $\CCC_\delta$ denote the subset 
of elements $K$ of $\CCC$ such that $\hbox{diam}\,K\ge \delta$; this set is  finite since $\CCC$ is a null-sequence
and non-empty for small enough $0<\delta \le m$. 

For all $K\in\CCC$, we can find two points 
$x_1,x_2\in K$ and a group element $g\in G$ such that $\{g(x_1),g(x_2)\}$ is $m$-separated: this implies
that $g(K)\in\CCC_m$, so that $\CCC$ is composed of finitely many orbits.

Let $K\in\CCC$ be a perfect compact set. Since $G_K=\Stab\ (K)$ is a subgroup of $G$, its action on the set of distinct triples of $K$ is automatically properly discontinuous.
Let us prove that it is also geometrically finite.  

Let $x,y\in K$, $x\neq y$, and assume that $x$ is conical for $G$. Let $(g_n)$ be a sequence of $G$ such that
$(g_n(x))_n$ tends to $a$ and $(g_n(y))_n$ tends to $b\ne a $.   This means that for all $n$ large enough $\hbox{diam}\,g_n(K)$ is larger than some constant $\delta>0$ (for instance $\delta=d(a,b)/2$)
so belongs to a finite subcollection of $\CCC$.   Extracting a subsequence if necessary, we may assume that $g_n(K)=L$ for some $L\in\CCC$. 
It follows that $h_n= g_1^{-1}g_n$  defines a sequence of $G_K$ such that $(h_n(x))$ tends to $g_1^{-1}(a)$ and $(h_n(y))$ tends to $g_1^{-1}(b)$ for all other
points $y$. This means that $x$ is conical for $G_K$.

 If $x\in K$ is parabolic, denote by $G_x$ its stabilizer and let $L$ be a compact fundamental domain for the action of $G_x$
on $\partial(G,\PP) \setminus\{x\}$. We first prove that $G_x\cap G_K$ is infinite, establishing that $x$ is a also a parabolic point for $G_K$. Since $x$ is non-isolated in $K$, we may find a sequence $(x_n)_n$ in $K$ which
tends to $x$ and a sequence  $(g_n)$ in $G_x$ so that $g_n(x_n)\in L$. The collection $(g_n)_n$ is infinite 
and  $\hbox{diam}\,g_n(K)$ is at least $d(x,L)>0$
so $g_n(K)$ belongs to a finite subcollection  $\CCC_L$. Extracting a subsequence if necessary, we may assume that $g_n(K)$ is a fixed  compact subset so that $(g_1^{-1}g_n)_n$ is an infinite
sequence in $G_x\cap G_K$. We will now prove that $x$ is also bounded as a parabolic point for $G_K$.  Let us label the elements of $\CCC_L$  by  $\{K_1,\ldots, K_N\}$ and let us  fix, for each index $j\in\{1,\ldots, N\}$,
an element $h_j\in G_x$ such that $h_j(K)=K_j$. Set $L_K= \cup_{1\le j\le N} h_j^{-1}(L)$ that is compact in $\partial(G,\PP) \setminus \{x\}$. 
For any $y\in K\setminus\{x\}$, we may find $g\in G_x$ so that $g(y)\in L$; note that $g(y)\in K_j$ for some $j\in\{1,\ldots, N\}$, implying that $h_j^{-1}g(y)\in L_K$. This shows that $x$
is a bounded parabolic point. Thus, any point in $K$ is either conical or bounded parabolic for $G_K$, so that $G_K$ is geometrically finite with limit set $K$.
\end{proof}

Another group feature we can identify from the relatively hyperbolic boundary is splittings of the group. We will express a splitting of a group in terms of an action of the group on a simplicial tree with finite edge stabilizers and without edge inversions.
A splitting is said to be {\it relative}
to a certain collection of subgroups if every subgroup in this collection
fixes a vertex of the tree. It is {\it non-trivial} if no vertex of the tree is
fixed by the whole group. 

Given a group $G$ with an action on a simplicial tree $T$ with no edge inversions, we let $\Gamma=T/G$ be the orbit space. For each vertex $v$ of $\Gamma$ we may consider  a vertex group $G_v$  defined
as the stabilizer of a representative of the vertex in $v$. In the same manner, we define edge groups $G_e$ for edges. The action of $G$ on $T$ provides us with 
 injective maps $\phi_{0,e}: G_e \rightarrow G_v$, $\phi_{1,e}:G_e \rightarrow G_v$  defined whenever $e(0)$ or $e(1)$ is  $v$.
 
 We say the tuple $\mathcal{G}= (\Gamma, \lbrace G_v \rbrace, \lbrace G_e \rbrace, \lbrace \phi_{\epsilon, e} \rbrace)$ is a graph of groups, and $G$ is
  the {\it fundamental group of the graph of groups} $\mathcal{G}$. The set of generators of $G$ is the union of the sets of generators for all the $G_e$ and the $G_v$, together with a set containing a generator $t_e$ for each edge of $\Gamma$. The relations are all the relations in each $G_e$ and $G_v$, $t_e = 1$ if $e$ is in a fixed maximal tree, $t_e^{-1} = t_{\bar e}$ and $t_e \phi_{0,e}(x) t_e^{-1} = \phi_{1,e}(x)$ for all $x \in G_e$.

  \bigskip
  
If the Bowditch boundary of a relatively hyperbolic group consists of more than one component, then Bowditch showed that the group must split. More precisely the following holds.

\begin{theorem}\cite[Theorem 10.1]{bowrelhyp} \label{thm:bow10.1}
The boundary $\partial \Gamma$  of a relatively hyperbolic group, $\Gamma$, is connected if and only if $\Gamma$ does not split non-trivially over any finite subgroup relative to peripheral subgroups. \end{theorem} 

In the case where the group splits, we have the following description which is again due to Bowditch. 

\begin{theorem} \cite[Theorem 10.3]{bowrelhyp} \label{thm:bow10.3}
Suppose a relatively hyperbolic group pair splits as a graph of groups with finite edge groups and relative to the peripheral subgroups. Then each vertex group is hyperbolic relative to the peripheral subgroups that it contains
and its boundary  is naturally identified as a closed subset of the boundary of the whole group.\end{theorem} 

The following proposition 
is an immediate consequence of our Proposition \ref{prop:nullseq} and the above discussion and Theorems~\ref{thm:bow10.1} and \ref{thm:bow10.3} by Bowditch. 

\begin{proposition}\label{prop:infends} The set of components of the Bowditch boundary of a relatively hyperbolic group $(G,\PP)$ 
forms a null-sequence. Moreover, for each 
component containing at least two points, the stabilizer is hyperbolic relative to conjugates of the original
peripheral subgroups $\PP$. \end{proposition}

While the boundary of a relatively hyperbolic group is not always connected and sometimes contains cut points, the structure of cut points allows us to rule out a dendrite boundary: 

\begin{lemma}\label{cor:nodendrite} 
Let $(G, \mathcal{P})$ 
be a geometrically finite convergence group. 
Then $\partial(G, \mathcal{P})$ is not a dendrite.
\end{lemma} 

Recall that a \emph{dendrite} is a connected, locally connected, compact metric space containing at least two points that
admits no simple closed curve.

\begin{proof}
 According to \cite[Theorem 1.1]{Dasgupta:Hruska}, every cut point of $\partial(G, \mathcal{P})$ is a parabolic point. This readily implies that there are at most countably many cut points in $\partial(G, \mathcal{P})$.
 To reach the desired conclusion, we shall show that a dendrite contains an open path of cut points. To see this, let $L$ be a dendrite. Then $L$ is path connected according to \cite[II.5.1]{whyburnantop}, since it is a locally connected complete metric space. Let $x$ and $y$ be distinct points in $L$ and $p$ a path between them. Remove a point $z$ on $p$.  If $z \in L$ is a not a cut point, $L \setminus \{z\}$ is connected. Since $L\setminus\{z\}$ is connected and locally compact, i.e. a \emph{generalized continuum}, the fact that it is locally connected implies that it is path-wise connected \cite[II.5.2]{whyburnantop}. 
 Thus there is another path $p'$ from $x$ to $y$ that misses $z$. Then the set $\{r\in[0,1] \mid p'(r) \in p([0, 1])\}$ is closed in $[0, 1]$ and not all of $[0, 1]$ so its complement contains an open interval, and this gives us a loop in $L$, which is absurd.
\end{proof}

Lemma \ref{cor:nodendrite} can also be derived using Theorem   1.2 of \cite{Dasgupta:Hruska}. 

\medskip

The
statement of the next proposition is due to Susskind and Swarup for geometrically finite Kleinian groups \cite[Thm 3]{susskind:swarup}. 
The same argument applies to relatively hyperbolic groups.

\begin{proposition}[Susskind and Swarup]\label{suswarup} 
Let $(G,\PP)$ be relatively hyperbolic and $H,K$ be two
relatively quasiconvex subgroups. Then  $H\cap K$ is relatively quasiconvex and
$ \Lambda_H\cap \Lambda_K= \Lambda_{H\cap K} \cup P$ where $P$ is a (possibly empty) 
discrete set of common parabolic points.\end{proposition}

See \cite[Theorem 1.2(2)]{Hruskasurvey} for the result that the intersection of two relatively quasiconvex subgroups is relatively quasiconvex, and \cite{Yang} for the result about the limit set.

\medskip

Together with Theorem \ref{thm:bow10.3} above, we will rely on one more result regarding splittings, again by Bowditch.

\begin{theorem}\cite[Theorem 10.2]{bowrelhyp}\label{thm:bow10.2}
Any relatively hyperbolic group pair can be expressed as the fundamental group of a finite graph of groups with finite edge groups and with every peripheral subgroup conjugate into a vertex group, with the property that no vertex group splits non-trivially over any finite subgroup relative to the peripheral subgroups.
\end{theorem}

We obtain in this way

\begin{cor}\label{cor:bow} Suppose $(G,\mathcal{P})$ is a relatively hyperbolic group pair and  $\partial(G, \mathcal{P})$ is a Cantor set. 
The group $G$ is the fundamental group of a finite graph of groups  where all the edge groups are finite, and each vertex group is either finite or a peripheral group.  
\end{cor}

\begin{proof} We apply Theorem \ref{thm:bow10.2} to express $G$ as the fundamental group of a finite graph of groups with finite edge groups and with every peripheral subgroup conjugate into a vertex group, with the property that no vertex group splits non-trivially over any finite subgroup relative to the peripheral subgroups. Theorem \ref{thm:bow10.1} tells us that this graph of groups is non-trivial since the boundary
is disconnected. Since it is totally disconnected, Theorem \ref{thm:bow10.3} implies that a vertex group is the stabilizer of a point or 
{finite}, so each vertex group is either conjugate to a peripheral subgroup or  finite. 
 \end{proof}

\begin{theorem} \label{thm:Cantor} Let $(G, \mathcal{P})$ be a relatively hyperbolic group pair with $\partial(G, {\mathcal{P}})$ homeomorphic to a Cantor set.  Assume that for all $P \in {\mathcal{P}}$, $P$ is residually finite.  Then $G$ is virtually $F *(*_1^n P_i)$ where $F$ is free (possibly of rank 0) and each $P_i$ is a finite index subgroup of some $P \in \mathcal{P}$. \end{theorem}

Theorem \ref{thm:Cantor} follows from Corollary \ref{cor:bow} and 

\begin{theorem} \label{thm:GOG} Let $(G,\mathcal{P})$ be a relatively hyperbolic group pair, such that $G$ can be written as a finite graph of groups, where every edge group is finite and each vertex group is either finite or a peripheral group.  Assume that each peripheral group is residually finite.  Then $G$ is virtually the free product of a free group and finite index subgroups of peripheral groups. 
\end{theorem}

 Let $G$ be the fundamental group of a graph of groups $\mathcal{G}$ with underlying graph $\Gamma$.  Suppose further, as in the hypotheses of Theorem \ref{thm:GOG}, that  $\mathcal{P}$ is a collection of subgroups of $G$ where each subgroup is residually finite,  each edge group is finite, and each vertex group is either finite or a subgroup in $\mathcal{P}$.

\begin{proof}[Proof of Theorem \ref{thm:GOG}]
We will map $G$ to the fundamental group of a graph of groups $G'$, over the same graph $\Gamma$ but where the vertex groups and edge groups are all finite. 
For each infinite vertex group $G_v$, conjugate to some $P \in \mathcal{P}$, there are finitely many edges meeting the vertex $v$.  Since $P$ is residually finite, there is a map
 $\psi_v:G_v  \rightarrow C_v$ onto a finite group $C_v$
 which is injective on the union of the images  $\phi_{\epsilon, e}(G_e)$ where $e(\epsilon) =v$.  
We will define $G'$ as the fundamental group of $\mathcal{G'} = (\Gamma, \lbrace G'_v \rbrace, \lbrace G_e \rbrace, \lbrace f_{\epsilon, e} \rbrace)$ where \begin{itemize} 

\item $G'_v = G_v$ if $G_v$ is finite, and $G'_v = C_v$  if $G_v$ is infinite. 
\item $f_{\epsilon, e} = \psi_v\circ \phi_{\epsilon,e} : G_e \rightarrow C_v$ if $G_{e(\epsilon)} $ is infinite, and $f_{\epsilon, e} = \phi_{\epsilon,e}$ if $G_{e(\epsilon)}$ is finite. 
\end{itemize} 

The group $G$ admits a natural surjection to $G'$.  Furthermore, $G'$ admits a surjection to a finite group which is injective on every edge and vertex group of $G'$, by Scott and Wall \cite[Chapter 7]{scott:wall}.  
Then the composition of these two maps is a map from $G$ to a finite group which is injective on every finite vertex group and every edge group.  The kernel $H$ of this composition is a finite index subgroup
of $G$ which acts on the  same tree as $G$ but with trivial edge groups. Thus, we see $H$ as the fundamental group of a finite graph of groups where the edge groups are trivial and the vertex groups
are either trivial or finite index subgroups conjugate to peripheral subgroups of $G$. This implies that $H$ is the free product of a free group and finite index subgroups of peripheral groups. 
\end{proof}

\section{Geometrically finite convergence groups acting on $S^2$}
\label{sec:S2}

A relatively hyperbolic group pair $(G,\PP)$ can have a planar boundary where the action does not extend to $S^2$;  see, for example \cite[Section 9]{KK}, where the group $G$ is hyperbolic and virtually Kleinian. The group $G$ need not be virtually Kleinian for $\partial(G, \PP)$ to be planar, though, and its peripheral subgroups can be arbitrary \cite{hruska:walsh}.
Here we collect some general results on geometrically finite convergence groups on $S^2$, which will be used for the more specific case of Schottky sets which we study here. 

Let $G$ be a convergence group acting on $S^2$ with limit set $\Lambda=\Lambda_G\subset S^2$. A relatively hyperbolic group pair $(G,\PP)$ is a {\it geometrically finite convergence group on $S^2$} if every point of $\Lambda$ is either a bounded parabolic point (with maximal parabolic group in $\PP$) or a conical limit point.  
We are not in general assuming that the action is faithful: there could be a finite normal subgroup of $G$  which acts as the identity on $S^2$.  When we know that the quotient by this finite normal subgroup is virtually a 2 or 3-manifold group, there is a finite index subgroup of $G$ which acts  as a subgroup of $Homeo(S^2)$, by \cite[Theorem 1.3]{peter:cyril}.  In what follows we will be analyzing the quotients by the finite normal subgroup, and the results in general will be virtual.

\begin{lemma}\label{lma:translation} An infinite-order, orientation-preserving parabolic element of a geometrically finite convergence group on $S^2$ is conjugate to a translation {in the plane obtained by removing from $S^2$ the fixed point of the parabolic element}.
\end{lemma}

\begin{proof} Let $g\in G$ be parabolic with fixed point $p\in S^2$. Its restriction to the plane $S^2\setminus \{p\}$ is fixed-point
free and its action is properly discontinuous. Hence, $(S^2\setminus\{p\})/\langle g\rangle $ is a surface with cyclic  
fundamental group and so homeomorphic to a cylinder. This implies that the action of $g$ is conjugate to that
of a translation in the plane. \end{proof}

\begin{proposition}\label{prop:planarparabolics} 
 Let {$(G, \mathcal{P})$} be a geometrically finite convergence group on $S^2$ with $G$ finitely generated. Then each $P \in \PP$ is virtually a finite type surface group, that is virtually free of rank at least $1$ or virtually a closed surface group.
\end{proposition}

\begin{proof}
{Since we are assuming that $G$ is finitely generated, according to \cite[Prop.~2.29]{Osin06} any maximal parabolic subgroup $P$ of $G$ is also finitely generated.}
Since $P$ also acts properly on $\mathbb{R}^2$, {the conclusion follows from} 
\cite[Cor. 3.2]{hruska:walsh}  {whose } 
proof uses \cite[Thm 1.3]{peter:cyril} in the case that {the action is not faithful.} 
\end{proof}

{\noindent\bf Elementary action.---} A convergence action on a compact metrizable space is {\it elementary} if its limit set is finite, i.e.,
contains at most two points.  Such actions are classified on the sphere, cf. \cite[Theorem 3.4, Lemma 4.2]{martin:skora}.

\begin{proposition}\label{prop:elementary} Let $G$ be a finitely generated subgroup of $Homeo^+(S^2)$ (the orientation preserving homeomorphisms of $S^2$) which is an elementary convergence group. 

{Then $G$ is isomorphic} to a subgroup of M\"obius transformations.  More precisely,
\begin{enumerate}
\item If $\Lambda_G=\emptyset$, then its action is conjugate to that of a finite subgroup of $SO_3(\R)$;
\item If $\Lambda_G$ is a singleton, then $G$ contains a finite index subgroup whose action is conjugate {on $S^2\setminus\Lambda_G$} either to that of a discrete translation group acting on the plane or to a Fuchsian group {acting on the unit disk} that defines a surface of finite type. In particular, if $G$ is two-ended, then $G$ is either isomorphic to $\Z$ or to $(\Z/2\Z)*(\Z/2\Z)$.
\item If $\Lambda_G$ is a pair of points, then $G$ has an index 1 or 2 subgroup $G'$ which is Abelian, and the action of $G'$ is conjugate to that of $\langle z\mapsto 2z, z\mapsto \zeta z\rangle$, with $\zeta^n=1$ for some $n\in{\mathbb{N}}$.
\end{enumerate}
\end{proposition}

\begin{proof} If the limit set is empty, then the action of $G$ on $S^2$ is properly discontinuous and cocompact so that $S^2/G$ is naturally equipped with a good spherical
orbifold structure. In other words,  its action is conjugate to that of a finite subgroup of $SO_3(\R)$.

Let us now assume that the limit set consists of a single point $p$. It must be parabolic so 
Proposition~\ref{prop:planarparabolics} implies that it is virtually a surface group of finite type. 

Let us now assume furthermore that $G$ is two-ended. Following \cite[Theorem 5.12]{scott:wall}, we consider a finite normal subgroup $F$ of $G$. Note that, as $F$ is normal and finite, we can find an infinite order element $g$ in $G$ that centralizes it. This implies that $F$ also fixes the point $p$, and, by the previous case, $F$ has to be a finite cyclic group. 
Let us assume that $F$ is non trivial and let $q$ be the other fixed point under $F$. Since $g$ has
infinite order, it acts as a translation by Lemma \ref{lma:translation}. Thus we should have $g^n f g^{-n}(q)=f(q)=q$ for all $n\in\Z$ and $f\in F$. However, this shows that $f$ fixes infinitely many points and cannot be non-trivial, 
so that we may conclude that $F$ is 
trivial and that $G$ is isomorphic to $\Z$ or $(\Z/2\Z)*(\Z/2\Z)$ by \cite[Theorem 5.12 (iii)]{scott:wall}. 

We now assume that $\Lambda_G$ has two points, so $G$  is two-ended. Now take the subgroup $G'$ of $G$ which fixes pointwise $\Lambda_G$. This is a subgroup of index at most two. As above, we consider a finite normal subgroup $F$ of $G'$. 
Since $F<G'$ fixes $\Lambda_G=\{p,q\}$ pointwise and is finite, $F$ has to be a rotation group, 
i.e., a finite (cyclic) subgroup of $SO_2(\R)$.
Since the action is properly discontinuous, cocompact and free on $S^2\setminus\{p,q\}$, the quotient by $G'$
is a torus. If $G' \neq G$, $G = G' \rtimes \mathbb{Z}/2\mathbb{Z}$, where $\mathbb{Z}/2\mathbb{Z}$ acts dihedrally on $G'$. The result follows.
 \end{proof}

The following is immediate from \cite{Dasgupta:Hruska}, which extends \cite[Thm\,0.1]{bowbound} and \cite[Thm\,0.2]{bowconnect}.

\begin{corollary}\label{cor:lc}
Let $(G, \mathcal{P})$ be a geometrically finite convergence group acting on $S^2$ with relatively hyperbolic boundary $\partial(G, \mathcal{P})$.  Then each component of $\partial(G, \mathcal{P})$ is locally connected and has the property that every cut point  is a parabolic point.
\end{corollary} 

\begin{lemma}\label{lma:omegafuchs} Let $(G,\PP)$ be a geometrically finite convergence group on $S^2$ with connected Bowditch boundary $\partial(G, \PP) = \Lambda$.
The ordinary set $S^2\setminus\Lambda$ is made of finitely many orbits of components, each with stabilizer which is a virtual 2--orbifold group. When $G$ is finitely generated, these are virtually finite-type surface groups.  
\end{lemma}

\begin{proof} 
We start by observing that the compact connected set $\Lambda$ is also locally connected by Corollary~\ref{cor:lc}. According to \cite[Theorem VI.4.4, p. 113]{whyburnantop}, local connectivity of $\Lambda$ assures that the components of the ordinary set $S^2\setminus\Lambda$ form a null-sequence ($\Lambda$ is an $E$-set) and that the boundary of each component is locally connected.
Moreover, each component $\Omega$ is simply connected. This follows from the fact that every simple closed curve contained in the surface $\Omega$ separates the sphere into two disks, one containing the connected set $\Lambda\supset \partial\Omega$ and the other contained in $\Omega$. 

If $S^2\setminus\Lambda$ has only finitely many components, then $G$ contains a finite index subgroup that stabilizes each component, i.e., a relatively quasiconvex subgroup. Of course, in this case $S^2\setminus\Lambda$ is made of finitely many orbits.

If $ S^2 \setminus\Lambda$ has infinitely many components forming a null sequence, we may apply part 2 of Proposition \ref{prop:nullseq} to conclude that $S^2\setminus\Lambda$ is made of finitely many orbits and their boundaries are stabilized by relatively quasiconvex subgroups. We claim that the stabilizer $H$ of a component $\Omega$ of $S^2\setminus\Lambda$ is of finite index (at most 2) in the stabilizer of its boundary $\partial\Omega$. This shows that $H$ is relatively quasiconvex. The claim follows by observing that the elements of the stabilizer of $\partial\Omega$ that do not leave $\Omega$ invariant must permute the components of $S^2\setminus\Lambda$ that have the same boundary as $\Omega$. If $\Omega$ is a Jordan domain, that is, if its closure is an embedded disk, then its boundary is a Jordan curve and either bounds one or two 
components of $S^2\setminus\Lambda$.  If $\Omega$ is not a Jordan domain, the boundary of $\Omega$ is not an embedded circle. However, since $\partial\Omega$ is locally connected  (cf. the first paragraph of the proof), the Carath\'eodory-Torhorst theorem applies: we can find a homeomorphism of the open disk onto $\Omega$ which extends continuously to the boundaries, {$h: \overline{D^2} \rightarrow \overline{\Omega}$}. Since $\partial\Omega$ is not a circle, this map cannot be injective. We can thus find a simple closed curve $\gamma$ which is contained in the closure of $\Omega$ and meets $\partial\Omega$ in a single cut point. The curve $\gamma$ is the image of a simple arc joining two points of the boundary of the closed disk which are mapped to the same point in $\partial\Omega$. Since every other component of $S^2\setminus\Lambda$ must sit on one side or the other of $\gamma$, another component cannot have the same boundary as $\Omega$. We thus see that $H = Stab(\Omega)$ coincides with the stabilizer of $\partial\Omega$ in this case.  Since $\Omega$ is an open disc and $H$ acts properly discontinuously on $\Omega$, $H$ is  virtually an orbifold group by \cite[Cor. 3.2]{hruska:walsh}.

When $G$ is a finitely generated convergence group acting on $S^2$, the peripheral subgroups are finite-type orbifold subgroups by Proposition \ref{prop:planarparabolics}. We claim that the peripheral subgroups of $H$ are finitely generated, hence $H$ is finitely generated, since it is relatively quasiconvex, and hence hyperbolic relative to the induced peripheral subgroups.   Any peripheral subgroup $Q$ of $H$ is $P\cap H$, where $P$ is a peripheral subgroup of $(G, \mathcal{P})$.  This is exactly the subgroup of $P$ that takes $\Omega$ to itself.  Then $\Omega/Q$ embeds in the finite-type orbifold $(S^2 \setminus \{p\})/P$, so is an embedded sub-orbifold of a finite-type orbifold, and hence of finite type.  Since the peripheral subgroups are finitely generated, so is $H$. 
\end{proof}
 
 Let $(G, \mathcal{P})$ be a geometrically finite convergence group with Bowditch boundary $\partial(G, \PP) = \Lambda$. By Proposition \ref{prop:infends}, we may assume that $\Lambda$ is connected.  Let $H<G$ be the stabilizer of a component $\Omega$ of $S^2\setminus \Lambda$. By  Lemma \ref{lma:omegafuchs}, we may find a homeomorphism $f:\mathbb{D}\to \Omega$, where $\mathbb{D}\subset \mathbb{C}$ is the unit disk in the complex plane, that conjugates the action of $H$ on $\Omega$ to that of a Fuchsian group $F$ on  $\mathbb{D}$, in the sense that 
 $f^{-1}\circ H\circ f=F$ on $\mathbb{D}$. The action of $F$ extends to the whole Riemann sphere. Since $\partial\Omega$ is locally connected, the homeomorphism $f$ extends to a continuous map $h:\overline{\mathbb{D}}\to \overline{\Omega}$ that semi-conjugates $F$ to
$H$, i.e., $hF={H h}$ on $\overline{\mathbb{D}}$.

\begin{lemma}\label{lma:parabolicvsfuchs}  Let $(G, \PP)$ be a geometrically finite, non elementary, convergence group on $S^2$ with limit set $\Lambda$.  Let $\Omega$ be a simply connected component of $S^2 \setminus \Lambda$ and {$f:\mathbb{D}\to \Omega$ the homeomorphism that conjugates the stabilizer $H$ of $\Omega$} to a Fuchsian group, and $h:\overline{\mathbb{D}}\to\overline{\Omega}$ the extension of this homeomorphism as above. Let $p\in\partial\Omega$ be a parabolic point with stabilizer $P$, and set {$Q= f^{-1}\circ(P\cap H)\circ f$}. 
Then the limit set $\Lambda_Q$ is exactly the non-empty set  $h^{-1}(\{p\})$.
\end{lemma}

\begin{proof}  First use Proposition \ref{prop:infends} to reduce to the case when the boundary is connected. 
Let $K\subset S^2\setminus\{p\}$ be a compact subset containing a fundamental domain for the action
of $P$ on $\Lambda\setminus\{p\}$.  We may find
$g\in P$ such that $g(\partial\Omega)\cap K\ne \emptyset$. 
Thus, $g(\Omega)$ contains in its closure
the point $p$ and at least one point of $K$. Such components form a finite set since $\Lambda$ is an $E$-set (Def. \ref{def:nses}). 
By considering a sequence of points in $\partial\Omega$ tending to $p$, we may pick an infinite sequence $(g_n)$ in $P$ that
maps $\Omega$ to components whose closures intersect both $\{p\}$ and $K$.
As there are only finitely many of them,  we may assume that $g_n(\Omega)=V$ for a fixed component 
$V$,
and all $n\ge 1$. Therefore, $(g_1^{-1}g_n)_n$
is an infinite collection of elements of $H\cap P$ which proves that $\Lambda_Q$ is not empty.

Since $hQ\subset Ph$, it follows that $h^{-1}(\{p\})$ is $Q$-invariant and compact, hence it contains
$\Lambda_Q$ which by definition is the minimal  compact invariant subset under the action of $Q$. 

The equality will follow from the fact that $p$ is a bounded parabolic point. We first rule out the case that $h^{-1}(\{p\})$ contains
an interval. If this were the case, then it would be the whole circle by invariance so that we would have $H=P$; this
contradicts that $\Lambda_P=\{p\}$ and $\Lambda_H=\partial\Omega$. 
Therefore, $h^{-1}(\{p\})$ is nowhere dense in $\mathbb{S}^1$.

Let $\Omega_1,\ldots \Omega_k$, be the $P$-translates
of $\Omega$ whose closures intersect both $\{p\}$ and $K$ and let us fix $g_1,\ldots g_k\in P$ such that
$g_j(\Omega_j)= \Omega$. Set $L= h^{-1}(\cup_{1\le j\le k} g_j(K))$.  This is a compact subset of $\overline{\mathbb{D}}$ disjoint from $h^{-1}(\{p\})$, hence from $\Lambda_Q$.
Let $x\in h^{-1}(p)$.  
We want to prove that the action of $Q$ is not equicontinuous at $x$.
With that in mind,
pick a point $y\in\mathbb{S}^1\setminus h^{-1}(\{p\})$ arbitrarily close to $x$. 
Note that $h(y)\in\partial\Omega\setminus\{p\}$ and since $K$ is a fundamental domain, we may find $g\in P$ and $j\in\{1,\ldots, k\}$ so that $g(h(y))\in K\cap\partial\Omega_j$.
It follows that $g_jg\in (H\cap P)$ so that we may find $q= h^{-1}g_jgh\in Q$ with $q(y)\in L$. This implies that $x\in\Lambda_Q$. 
Indeed, considering now a sequence $(y_n)$ in $\mathbb{S}^1\setminus h^{-1}(\{p\})$ tending to $x$, we obtain in this way a sequence $(q_n)$ in $Q$ such that $(q_n(x), q_n(y_n))\in h^{-1}(\{p\})\times L$: as $L$ and  $h^{-1}(\{p\})$ are disjoint compact subsets, $(q_n)_n$ cannot be equicontinuous at $x$.
\end{proof}

As already observed in the proof of Lemma~\ref{lma:omegafuchs}, if $h$ is not injective, i.e., if $h(x)=h(y)$ for some pair of points $x,y$ in $\mathbb{S}^1$, then $h(x)$ is a cut point of $\Lambda_G$ (we may build a Jordan arc 
in $\overline{\mathbb{D}}$, a crosscut,  that maps under $h$ to a separating Jordan curve), and so $h(x)$ is parabolic. 

Given a parabolic point $p$ with stabilizer $P$ and a simply connected component $\Omega$ of the ordinary set which
contains $p$ in its boundary, we will say that $p$ is {\it uniquely accessible from $\Omega$} if the above
map $h:\overline{\mathbb{D}}\to\overline{\Omega}$ is injective over $p$, i.e., $h^{-1}(\{p\})$ is 
a singleton.  Likewise, we say that $p$ is {\it doubly accessible from $\Omega$} if $h^{-1}(\{p\})$ consists of two points. We expect that in general $h^{-1}(p)$ will be a Cantor set if $p$ is not uniquely or doubly accessible. 

\begin{corollary} \label{cor:rank1}  
Let $(G,\PP)$ be a geometrically finite convergence group acting on $S^2$ and let $p\in S^2$ be a parabolic point with stabilizer $P\in\PP$. 
Assume that the component of its Bowditch boundary containing $p$ is not a singleton. 
Let $\Omega$ be any simply connected component of $S^2 \setminus \partial(G, \mathcal{P})$ such that $\partial\Omega$ contains $p$. If $P$ is two-ended, then $p$ is either uniquely or doubly accessible from $\Omega$. 
\end{corollary} 

\begin{proof} Recall the notation of Lemma \ref{lma:parabolicvsfuchs}: $Q$ is defined as the Fuchsian group such that $h\circ Q=(P\cap H)\circ h$, 
where $h$ is the extension of the homeomorphism conjugating the action of the stabilizer $H$ of $\Omega$. The accesses to $p$ from $\Omega$ 
are in bijection with the points of $\Lambda_Q$.  By Lemma \ref{lma:parabolicvsfuchs}, the limit set is non-empty so $Q$ is infinite.

Since we assume that $P$ is two-ended, this is also the case of $Q$. Hence there is a finite index cyclic subgroup in $Q$ that is generated either
by a loxodromic element, implying the point $p$ is doubly accessible, or by a parabolic element, 
implying the point $p$ is uniquely accessible.
\end{proof}

We note that the converse of Corollary~\ref{cor:rank1} does not hold in full generality. 
Here is a counter-example: pick a convex-cocompact Kleinian group $G$ that uniformizes a hyperbolic 3-manifold with totally geodesic boundary; consider one component
$F$ of its boundary and choose a compact $\pi_1$-injective proper subsurface $S$ in $F$, with a non-Abelian free fundamental group $P$, such that each component of the complement of $S$ has also non-Abelian free fundamental group. The pair $(G,\mathcal{P})$, where $\mathcal{P}$ consists of the conjugates of $P$, is a planar relatively hyperbolic group pair.  
To see this, $P$ stabilises a component $\Omega_F$ of the ordinary set, hence the hyperbolic convex hull $K$ of $\Lambda_P$ in $\Omega_F$ is connected and simply connected in $\Omega_F$, and precisely invariant under $P$, i.e., if $g(K)\cap K\ne \emptyset$ for some $g\in G$, then $g\in P$. Therefore, as $\Lambda_G$ is a  Sierpi\'nski carpet, $G(K)$ is a null sequence that satisfies the assumptions of Moore's Theorem \ref{thm:moore}: by collapsing each component
of $G(K)$ we obtain a geometrically finite convergence group action on $S^2$ for which $P$ is parabolic with fixed point $p$. Moreover, the parabolic point $p$ is on the boundary of countably many components $\Omega$ such that $\Stab(\Omega)\cap P $
is cyclic but $P$ is not, and $p$ is uniquely accessible from each component. 

\begin{proposition}\label{prop:parabsubgp}  Let $p$ be a parabolic point with stabilizer $P$ of a geometrically finite convergence group on $S^2$, $(G,\PP)$. We assume that the component of $\partial(G, \PP)$ containing $p$ is not a singleton.
Let $\Omega_p$ denote the union of the 
components {of the ordinary set} which contain $p$ on their
boundary. The action of $P$ on $S^2\setminus (\{p\}\cup\Omega_p)$ is cocompact and
the set of components of $\Omega_p$ forms finitely many orbits.

In particular, if $p$ is in the boundary of no ordinary component, then $P$ acts cocompactly on $S^2\setminus\{p\}$.
\end{proposition}

\begin{proof} We may assume that  $\Lambda_G$ is connected according to Proposition \ref{prop:infends}.
Let $K\subset S^2\setminus\{p\}$ be a compact subset containing a fundamental domain for the action
of $P$ on $\Lambda_G\setminus\{p\}$.

Let $L$ denote the closure of the union of the components of the ordinary set disjoint from $\Omega_p$ whose boundaries intersect $K$.
Since $\Lambda_G$ is an $E$-set, we may deduce that $L$ is a compact subset of $S^2\setminus\{p\}$. 
{For} any component $\Omega$ disjoint from $\Omega_p$, we may find
$g\in P$ such that $g(\partial\Omega)\cap K\ne \emptyset$ {holds so that $g(\Omega)\subset L$}. It follows that the action is cocompact on 
$S^2\setminus(\{p\}\cup\Omega_p)$.

We now consider components $\Omega$ which contain $p$ on their boundary. As above, we may find
$g\in P$ such that $g(\partial\Omega)\cap K\ne \emptyset$. Thus, $g(\Omega)$ contains in its closure
the point $p$ and at least one point of $K$. Such components form a finite set since $\Lambda_G$ is an $E$-set.
\end{proof}

We conclude with some general properties of the ordinary set, which are proved as for Kleinian groups. The next proposition
was already known, but we were unable to find a formal proof in the literature. 

\begin{proposition}\label{prop:ordcomp} Let $G$ be a convergence group acting on $S^2$. Then the ordinary set has zero, one, two or infinitely many
components. Furthermore, if the ordinary set is non-empty, then the limit set $\Lambda_G $ has no interior.
\end{proposition}

\begin{proof}
If the action of the convergence group $G$ is elementary, then $\Omega_G$ is connected. Otherwise, let us assume that the action is non elementary and that $\Omega_G$ has at least two, but finitely many
components.

Considering a finite-index subgroup if necessary, one may assume that the group $G$ fixes each component. Therefore, $\Lambda_G$ is the 
boundary of each component of the ordinary set. This is the main point and follows from the fact that the boundary of each component is closed, contained in $\Lambda_G$, and $G$-invariant, since each component is $G$-invariant.

{As the limit set is infinite,} the group $G$ contains a loxodromic element $g$ with fixed points $a$ and $b$ in $\Lambda_G$. 

Consider a component $\Omega$ of the ordinary set and a point $x\in \Omega$. We may find a path $c_0$  in $\Omega$ that joins $x$ to $g(x)$. The $g$-orbit 
$c_n=g^n(c_0)$, $n\in\Z$, defines a path which joins $a$ and $b$ in $\Omega$ by the convergence property. Its image contains an arc $c$ also  joining $a$ and $b$, i.e., a path without self-intersections.

Since, as remarked above, $\Lambda_G$ is the boundary of every component, we can proceed similarly with a second
component $\Omega'$ and denote by $c'$ an arc in $\Omega'$ which joins $a$ and $b$. Then $\{a,b\}\cup c \cup c'$ is a Jordan curve that separates $\Lambda_G$, for there are points of both $\Omega$ and $\Omega'$ on each side of the Jordan curve.
If there were a third component in the complement of $\Lambda_G$ it would sit on one side of this Jordan curve $\{a,b\}\cup c \cup c'$. Therefore, the boundary of this new component could not be $\Lambda_G$, which is absurd. 
Therefore, if there are more than two components, there must be infinitely many. 

To prove the last sentence, suppose that the ordinary set is non-empty and there is an open disc in $\Lambda_G$ around some point $x$ of $\Lambda_G $.  Take a point $y$ in $\Lambda_G$ which is in the closure of some component of the ordinary set.  By \cite[Theorem 2S]{tukia:convergence_groups}, $x$ is in the closure of the orbit of $y$ which is a contradiction. 
\end{proof} 

\begin{corollary} \label{cor:1or2} Let $(G,\PP)$ be a geometrically finite convergence group acting on $S^2$. If $\Omega_G$ is non empty and connected, then $\Lambda_G $ is totally disconnected. If furthermore $G$ is finitely generated, then $G$ is covered by a Kleinian group.  If $\Omega_G$ has exactly
two components, then $\Lambda_G$ is a circle and the action of $G$ is either isomorphic to a Fuchsian group of finite coarea, or to a degree 2 extension of such a group. \end{corollary}

\begin{proof} Let us assume that $\Omega_G$ is connected and let us assume for contradiction that $\Lambda$ is a component of the limit set with at least two points. Then $\Lambda$ is the limit set of 
its stabilizer $H$ which is also hyperbolic relative to virtual surface groups, cf. Proposition \ref{prop:infends}. Since $\Lambda$ 
does not separate the plane and does not contain an open disk, it is
simply connected. It follows that $\Lambda$ cannot contain a simple closed curve and, since it is locally connected by Corollary~\ref{cor:lc}, it is a dendrite, which is impossible by Lemma~\ref{cor:nodendrite}.
So $\Lambda_G$ is totally disconnected. Now, when $G$ is finitely generated, it follows from \cite[Corollary 5.4]{martin:skora} that $G$ is covered by a Kleinian group.

Assume $\Omega_G$ has two components $\Omega_\pm$.  Taking an index 2 subgroup if necessary, we may assume that both components
are invariant under $G$ so that $\overline{\Omega_+} \cap \overline{\Omega_-}=\Lambda_G$, the minimal $G$-invariant set.   This implies that $\Lambda_G$ is  their common boundary and is connected by \cite[Cor.\,VI.2.11]{whyburnantop}, 
hence locally connected
by Corollary \ref{cor:lc}, and that $\Omega_\pm$ are simply connected.
Thus, by the Carath\'eodory-Torhorst theorem, there are continuous onto maps $\varphi_\pm:\overline{\mathbb{D}}\to \overline{\Omega_\pm}$ from
the closed unit disk that restrict to homeomorphisms between their interiors. Since both images of the unit circle coincide, 
reasoning as in the proof of Lemma~\ref{lma:omegafuchs}, we may conclude that $\Lambda_G$ is a Jordan curve. Finally Lemma~\ref{lma:omegafuchs} enables us to conclude in this case.
For more general results, see \cite{martin:tukia:invpairs}.
\end{proof}

\section{Blowing-up rank-one parabolic points}\label{sec:blow}

\begin{definition} \label{def:parabolicsplit} Let $(G,\PP)$ be a geometrically finite convergence group on $S^2$. We write
$\PP=\PP_1\sqcup \PP_2$ where $\PP_1$ consists of all stabilizers of rank 1 parabolic points, i.e., whose stabilizers are $2$-ended, and $\PP_2=\PP\setminus\PP_1$. \end{definition} 

\begin{theorem}\label{thm:blowup} Let $(G,\PP)$ be a geometrically finite convergence group on $S^2$
and
$\PP=\PP_1\cup \PP_2$ as in Definition \ref{def:parabolicsplit}.  Then there exists a geometrically finite convergence group action of $(G,\PP_2)$ on $S^2$
 and an equivariant degree-$1$ continuous map
$\phi:S^2\to S^2$ mapping the Bowditch boundary of $(G,\PP_2)$ onto that of $(G,\PP)$.
\end{theorem}

In other words, the geometrically finite action of $(G,\PP)$ is covered by the geometrically finite action of $(G,\PP_2)$.

Before giving the proof, we pause for some topological facts, starting with the following particular case  of Moore's theorem \cite{moore}.  

\begin{theorem}[Moore]\label{thm:moore} Let $\CCC$ be a pairwise disjoint collection of  compact and connected subsets of the sphere $S^2$ such that each $K \in \CCC$ is not a point.  Assume that each element has a connected complement
and the set $\CCC$ forms a null-sequence. Let $\sim$ be the equivalence relation generated by $x\sim y$ if there is some $K\in\CCC$ that contains $\{x,y\}$. Then
$Z=S^2/\sim$ is a topological sphere when endowed with the quotient topology. \end{theorem}

We note that the quotient map in Moore's theorem is a degree-$1$ map, as proven in cf. \cite[Theorems 13.4, 25.1]{daverman:decompositions}.

\medskip

We add some further properties that will be used in the proof of Theorem \ref{thm:blowup}.

\begin{proposition}\label{prop:Y}    Under the assumptions of Theorem \ref{thm:moore}, set $Y=S^2\setminus \cup_{K\in\CCC} K$.
For any connected open subset $U$ of $S^2$ such that $\partial U\subset Y$, the set $Y\cap U$ is arcwise connected. In particular
$Y$ 
 is arcwise connected, and every point of $Y$ admits a  basis of neighborhoods such that the boundaries of these neighborhoods are disjoint
from $S^2 \setminus Y = \cup_{K\in\CCC} K $.  \end{proposition}

\begin{proof} Denote by $\pi:S^2\to Z$ the canonical projection and note that 
for all $y\in Y$, $\pi^{-1}(\pi(\{y\}))=\{y\}$, in particular the restriction of $\pi$ to $Y$ is injective. 
We first justify that if $A$ is a compact arc or a Jordan curve in $\pi(Y)$, then so is $B=\pi^{-1}(A)$. To see this, note that $B$ is compact and that
$\pi:B\to A$ is bijective and continuous since $B\subset Y$. 

Since the projection $\pi:S^2\to Z$ maps $Y$ to the complement of a countable set,
we may find arcs joining
any two points in $\pi(Y)$ and then lift them back to $Y$. This proves that $Y$ is arcwise connected,  as well as $U\cap Y$ for any connected open 
set with $\partial U\subset Y$, for in this case $U$ is saturated and $\pi(U)$ is open.
Similarly, if $x\in Y$, then we may construct a basis of disk-neighborhoods of $\pi(x)$ in $Z$ with their boundaries contained in $\pi(Y)$. 
They lift as disk neighborhoods of $x$ in $S^2$. Since $\CCC$ is a null sequence and $x$ is disjoint from the collection $\CCC$, these disk-neighborhoods  form a basis. \end{proof}

{\noindent\bf Scheme of the proof of Theorem \ref{thm:blowup}.---}
The statement only refers to the action of $G$ on $S^2$ and not to the group itself. Therefore,  to prove the theorem it is sufficient to assume that $G$ acts faithfully, which we shall do from now on.

Our proof has two key parts. In the first part we blow-up the sphere $S^2$ over each parabolic fixed point of the subgroups $P\in\PP_1$. More precisely, we blow-up the complement of horoballs for the original action. This blown-up space admits a natural quotient map $\pi$ onto the complement of the horoballs. We then show that the blown-up space is planar and lift the action of $G$ by $\pi$ and we finally extend both $\pi$ and the action to the whole sphere. This is accomplished in several steps.

\medskip

 In {\bf Step 1}, we define and show properties of horoballs in the original action on $S^2$. We associate $G$-equivariantly to each parabolic point coming from $P\in \PP_1$ a pair of 
disjoint $P$-invariant horoballs. The set of these forms a null-sequence. Our next step is to blow-up the complement $Y$
of all these horoballs. For this, in {\bf Step 2}, we define our candidate blown-up space $\hat{Y}$ that comes from
a different completion of the complement $Y$ of the horoballs.  It comes with a continuous projection
$\pi:\hat{Y}\to \overline{Y}$, where $\overline{Y}$ is the closure of $Y$ in $S^2$. The topology of $\hat{Y}$ requires understanding. For this, in {\bf Step 3}, we give a basis of  
neighborhoods for points in $\overline{Y}$ that are well-behaved under $\pi$. In {\bf Step 4} we establish that
$\widehat{Y}$ is a planar, locally connected and arcwise connected compact
set. Furthermore, $\hat{Y}$ admits an embedding in the sphere such that its complementary components are open disks (Proposition \ref{prop:tophatY}). Moreover, the map $\pi:\widehat{Y} \to \overline{Y}$ is onto and $1:1$ except over the parabolic points where it is $2:1$ (Lemma \ref{lma:pifibers}).  We then embed $\hat{Y}$ in $S^2$. {\bf Step 5} extends the action of $G$ equivariantly. 
We show that the action of $(G,\PP)$ on $\overline{Y}$ lifts to a geometrically finite action of $(G,\PP_2)$ on $\hat{Y}$ (Proposition \ref{prop:liftedactiop}), using properties of  
Step 3.
This action is in turn extended to $S^2$ with the same properties. Then we extend the projection $\pi$ to $S^2$ equivariantly (Proposition \ref{prop:piext}). Theorem  \ref{thm:blowup} follows immediately from these propositions and the fact that for a geometrically finite action on $S^2$, the Bowditch boundary is 
homeomorphic to the limit set of the action. The formal proof of Theorem \ref{thm:blowup} from Propositions \ref{prop:liftedactiop} and \ref{prop:piext} is given at the end of the section. 

\medskip

{\noindent\bf Step 1. Definition and properties of horoballs.---} 
Let $\mathbb{P}_1$ denote a set of representatives of each conjugacy class in $\PP_1$. 

\medskip

Fix $P\in\mathbb{P}_1$  with parabolic point $p$.  We will  
define two disjoint horoballs in $\Omega_G$ attached to $p$. By a horoball, we mean the intersection with $\Omega_G$ of a closed Jordan domain in $\Omega_G\cup\{p\}$ that is invariant under  $g$ (cf. Section \ref{sec:general}). 
These will be two closed topological disks in $S^2$ which meet at the point $p \in \Lambda_G$.  
Let $P^+$ be the subgroup of index at most $2$ of $P$ of orientation-preserving elements.
Since it is two-ended, $P$ contains a subgroup $P_e$ of index at most $2$ which fixes both ends of $P$. 

According to Proposition~\ref{prop:elementary} 
the normal subgroup $T=P^+\cap P_e$ is generated by an element $g$ that acts as a translation cf. Lemma \ref{lma:translation}.

Let us consider a chart that identifies $S^2\setminus\{p\}$ with $\mathbb{C}$ and $g$ with the translation by $1$. Note that the action of $g$ on $\Lambda_G\setminus\{p\}$ is cocompact since
$g$ generates a finite index subgroup of $P$ and $p$ is a bounded parabolic point. Therefore, we may enclose $\Lambda_G\setminus\{p\}$ into a horizontal open strip of bounded width. 
The complement of the strip in $\mathbb{C}$ is the union of two half-planes contained in $\Omega_G$, each of which defines a closed horoball attached to $p$. 
Let $H_P$ denote their union. Since $P$ is the stabilizer of $p$ and $T$ is of finite index in $P$,  the two half planes can be chosen so that the stabilizer of $H_P$ is exactly $P$.
Choose the horoballs small enough so that the collection is pairwise disjoint in $\Omega_G$. This is possible since the action of $G$ is properly discontinuous on $\Omega_G$ and $\mathbb{P}_1$ is finite.

\medskip

\begin{lemma}\label{lma:Cnull} Let $\CCC$ denote the collection of all translates by $G$ of this finite collection, i.e., $\CCC= \{g\overline{H_P}, g \in G, P \in \mathbb{P} \} $, and $\overline{H_P}=H_P\cup\{p\}$.  The collection $\CCC$ forms a null sequence that is locally finite in $\Omega_G$.\end{lemma}

\begin{proof} We proceed by contradiction: we consider a sequence $(g_n)_{n \in \mathbb{N}}$ of $G$ such that $\hbox{\rm diam}\,g_n(H_P)\ge \delta$ for some $\delta>0$ and some fixed $P\in\mathbb{P}_1$ and associated point $p$. 
Up to taking a subsequence, by the convergence property, we may assume that $(g_n)$ tends uniformly towards 
the constant map with image $b\in\Lambda_G$
on the compact subsets of $S^2\setminus \{b'\}$, 
where $b,b'\in\Lambda_G$. 
We now remark that we must have $b'=p$. Indeed, if that was not the case, the closure of $H_P$ would be a compact set in the complement of $b'$, intersecting $\Lambda_G$ only in $p$. By the convergence property its images by the elements of the sequence should shrink to $\{b\}$, against the hypothesis that their diameter is bounded from below. Thus $\CCC$ is a null sequence.

Since additionally each element in $\CCC$ intersects $\Lambda_G$, it follows that the collection is locally finite in $\Omega_G$.
\end{proof}

{\noindent\bf Step 2. Definition of $\widehat{Y}$.---} Set $Y=S^2\setminus \cup_{K\in\CCC} K$, and  observe that that we are under the assumptions of Proposition \ref{prop:Y} by Lemma \ref{lma:Cnull}.  In particular, $Y$ is arcwise connected. 

 Endow $S^2$ with a distance $d_S$ compatible with its topology.
  We define, on the subset  $Y$ of $S^2$,  $d_Y(x,y)=\inf \hbox{\rm diam}_S L$ where
$L$ runs over all continua of $Y$ which contain $\{x,y\}$. This defines a metric on $Y$. Let us denote
by $\hat{Y}$ its metric completion. 

\medskip 

Let us observe that $d_S\le d_Y$ so that every Cauchy sequence for $d_Y$ is a Cauchy sequence for $d_S$. Therefore, 
this ensures the existence of a canonical continuous map $\pi:(\hat{Y},d_Y)\to (\overline{Y},d_S)$. 

\medskip

{\noindent\bf Step 3. Regular neighborhoods in $\overline{Y}$.---}
We define a notion of  regular neighborhoods for points in $\overline{Y} \subset S^2$. It will help us
understand the topology of $\hat{Y}$ from that of $\overline{Y}$ in $S^2$ under the projection $\pi$. 
\begin{itemize}
\item  By Proposition \ref{prop:Y},  every point in $y\in Y$ admits a basis of neighborhoods in $S^2$ 
whose boundaries are disjoint from the elements of $\CCC$. We call such neighborhoods {\it regular  of type $(Y)$} for $y$. 
\end{itemize}
Now let $K\in\CCC$ be the element associated to a rank 1 parabolic point $p$. Note that  $K$ is a union of two closed  disks attached at $p$
and that $K\setminus\{p\}$ is contained in $\Omega_G$, so isolated from the other components, cf. Lemma \ref{lma:Cnull}. 

\begin{itemize}

\item Let $x\in \partial K\setminus\{p\}$ and observe that $x\in \Omega_G$. It follows from the fact that $\CCC$ is a locally finite family in $\Omega_G$ that $x$ admits a basis of neighborhoods which are disks in $\overline{Y}$ bounded by the union of an arc in $\partial K$ and an arc in $\Omega_G\setminus K$.  We call such neighborhoods {\it regular of type $(K)$}  for $x$.

\item  For the point $p$,  we may consider a basis of Jordan disks whose boundaries  intersect $K$ in exactly two arcs, one in each horoball.
 We choose them to also be disjoint from elements of $\CCC\setminus\{K\}$ with the help of Proposition \ref{prop:Y} applied to $\CCC\setminus \{K\}$.
We call the intersection of such disks with $\overline{Y}$  {\it regular  of type $(P)$} for $p$.  \end{itemize}

By construction, if $V$ is a regular neighborhood of type $(Y)$ or $(K)$, then $Y\cap V$ is arcwise connected, whereas if $V$ is of type $(P)$, then $Y\cap V$  has exactly two arcwise
connected components of arbitrarily small diameter.

\medskip

{\noindent\bf Step 4. Properties of the set $\hat{Y}$.---} 
We now establish some properties of the set $\hat{Y}$. In particular we will see that $\hat{Y}$ embeds in $S^2$.

\begin{proposition}\label{prop:tophatY} The set $\widehat{Y}$ is a planar, locally connected and arcwise connected compact
set that admits an embedding in the sphere such that its complementary components are open disks.
\end{proposition}

We start with the following lemma.

\begin{lemma}\label{lma:pifibers} The map $\pi:\widehat{Y} \to \overline{Y}$ is onto and $1:1$ except over the parabolic points where it is $2:1$.
\end{lemma}

\begin{proof}
Let $(x_n)_n$ be a Cauchy sequence in $(Y,d_S)$ with limit $x\in\overline{Y}$. If $x$ is not a rank-one parabolic point, then it admits a basis of regular neighborhoods in $\overline{Y}$
of types $(Y)$ or $(K)$ that intersect $Y$ in an arcwise connected set, so that $(x_n)_n$ is also a Cauchy sequence in $(Y,d_Y)$  that defines a unique 
limit point in $\widehat{Y}$.
If $x$ is a rank-one parabolic point with stabilizer $P$, then $\overline{Y}\setminus\{x\}$ 
has two ends associated to the arcwise connected components of its regular neighborhoods of type $(P)$.
It follows that $\pi^{-1}(\{x\})$ has exactly two preimages,  one for each end. Thus, $\pi$ is also surjective and a point has two preimages if it is a parabolic point of rank one and one preimage otherwise. 
\end{proof}

Corollary \ref{cor:tophatY} will follow  from Lemma \ref{lma:pifibers} with the description of the regular neighborhoods in Step 3. 

\begin{corollary} \label{cor:tophatY}
The set $\hat{Y}$ is arcwise connected, locally connected, compact, 
with no local cut points 
and such that each component of $\hat{Y}\setminus Y$ is
 a non-separating simple closed curve  in $\hat{Y}$. 
\end{corollary}
\begin{proof} 
The metric space $\hat{Y}$ is connected as the completion of the connected space $Y$ by Proposition \ref{prop:Y}. 
Let us prove that $Y$ is compact. Let $(y_n)_n$ be a sequence in $\hat{Y}$. Since $\overline{Y}$ is compact, we may assume that 
$(\pi(y_n))_n$ is convergent to a point $x$. 
If $x$ is not a rank-one parabolic point, then $(y_n)_n$ tends to $\pi^{-1}(\{x\})$. If $x$
is a rank-one parabolic point, then we may consider a subsequence that remains in a single end of $\overline{Y}\setminus\{x\}$. This subsequence will provide a convergent subsequence in $\hat{Y}$. Compactness of $\hat{Y}$ follows.

We now prove local connectedness and absence of local cut points. Let $y\in \hat{Y}$ and  $x=\pi(y)\in \overline{Y}$. If $x$ is not a rank-one parabolic point, then, for any regular neighborhood $V$ of $x$, Lemma \ref{lma:pifibers} and Proposition \ref{prop:Y} imply that $\pi^{-1}(V)$ is a connected neighborhood of $y$ so that $\hat{Y}$ is locally connected at $y$ and regular neighborhoods of $x$ define a basis of connected neighborhoods of $y$. Moreover, $V\setminus\{x\}$ remains connected so that $y$ is not a local cut point.

If $x$ is a rank-one parabolic point, then its regular neighborhoods $V$ punctured at $x$ have two connected components $V_1\sqcup V_2\,(=V\setminus\{x\})$. 
However the point $y$ is defined by a class of $d_Y$-Cauchy sequences that are contained in a single component $V_j$. 
Thus, $\pi^{-1}(V_j)\cup\{y\}$ defines a connected neighborhood of $y$. 
Therefore, $\hat{Y}$ is also locally connected at $y$  and the connected components of punctured regular neighborhoods of $x$ contained in $V_j$
define a basis of connected neighborhoods for $y$. As $V_j$ is connected, it follows that $y$ is not a local cut point either.

Since $\hat{Y}$ is a connected locally connected compact metric space, it follows that $\hat{Y}$ is arcwise connected.

Let us now consider $y\in\hat{Y}\setminus Y$. Then there exists $K\in\CCC$ such that $\pi(y)\in \partial K$. Let $p\in K$
be its parabolic point. The set $\partial K\setminus \{p\}$ is a union of two arcs. By Lemma \ref{lma:pifibers}, they lift
as two arcs in $\hat{Y}$ that accumulate on both preimages of $p$. Thus $\pi^{-1}(\partial K)$ is a simple closed curve.

With the same notation, note that, by Proposition \ref{prop:Y}, $Y$ is a dense connected subset of $\hat{Y}\setminus \pi^{-1}(\partial K)$, so that the latter is connected as well. This implies that the components of $\hat{Y}\setminus Y$ are non-separating simple closed curves.
\end{proof} 

\begin{proof}[Proof of Proposition \ref{prop:tophatY}] If we prove that $\hat{Y}$ is planar then it will follow from the previous corollary that the complementary components are disks since the boundary curves are non-separating. Hence, 
it only remains to check that $\hat{Y}$ is planar. 
Claytor's theorem \cite{claytor} asserts that a continuum without local cut points is embeddable in the sphere  if and only if it contains neither a copy of the complete graph on five vertices $K_5$ nor of the complete bipartite graph  with six vertices $K_{3,3}$. 

Let us consider a finite connected 
graph $L$ and an embedding $j:L\hookrightarrow \hat{Y}$. We will modify the embedding $j$ so that $\pi\circ j$ is also injective, implying that $L$ cannot be one of the forbidden graphs.

Let $T\subset L$ denote the closure of the set of points $z\in L$ for which we may find $w\ne z$ in $L$ such that $(\pi\circ j) (z)= (\pi\circ j) (w)$. Note that $T$ is a compact subset of $L$.
If $T$ is empty, then there is nothing to do.
Let us assume it is not empty.
 Let $z\in T$. If $z$ belongs to an edge, we consider an open interval neighborhood $J_z\subset L$ contained in the same
edge; if $z$ is a vertex, then we consider a star-shaped open neighborhood $J_z$ contained in the union of the edges incident to $z$. Since $L\setminus J_z$
is compact, we have $d_Y(j(z), j(L\setminus J_z))>0$ so that we may find a regular neighborhood $\hat{V}_z\subset \hat{Y}$ of $j(z)$ such that $\overline{j^{-1}(\hat{V}_z) }\subset J_z$; we let
$V_z\subset \overline{Y}$ be the corresponding regular neighborhood of $(\pi\circ j)(z)$. 

We now extract a finite subcover of $(\pi\circ j)(T)$ given by the above regular neighborhoods that we order $V_1,\ldots, V_n$. Each $V_k$ comes with a point $z_k\in L$, a neighborhood $J_k\subset L$ and a regular
neighborhood $\hat{V}_k$ of $j(z_k)$. We modify the embedding $j$ inductively on the neighborhoods $V_k$. Let us fix $1\le k\le n$, and let us assume that $\pi\circ j$ is injective  on $L\setminus ( \cup_{k\le i\le n} J_{z_i} )$.
We note that $j(L)\cap \overline{\hat{V}_k}\subset j(J_k)$ and $\hat{V_k}\cap Y$ is homeomorphic to the complement of a countable subset of a Jordan domain. Therefore, we may modify $j|_{J_k}$ 
so that its image in $\hat{V}_k$ is contained in $Y$. As $\pi|_Y$ is injective, the map  $\pi\circ j$ is now injective  on $L\setminus ( \cup_{k< i\le n} J_{z_i} )$. 

In conclusion, given any embedding of a finite graph $L$ in $\hat{Y}$, there is an embedding of $L$ in $\overline{Y}$, hence in $S^2$. As the latter space is planar, we may conclude that $L$ is not isomorphic to
$K_5$ nor $K_{3,3}$, and so $\hat{Y}$ is planar. 
\end{proof}

\medskip

{\noindent\bf Step 5. Extension of the action of $G$.---} Let $\hat{y}\in \hat{Y}$, $y=\pi(\hat{y})$ and $g\in G$. The point $y$ is defined by a class of $d_Y$-Cauchy sequences in $(Y,d_Y)$ that enter an arcwise component of any regular neighborhood of $y$. These can be taken to be contained in a single component of a regular neighborhood of $y$ with $y$ removed, as in the proof of Corollary \ref{cor:tophatY} above.   Their images 
under $g$ also enter an arcwise component of regular neighborhoods of $g(y)$. Therefore, they also form equivalent $d_Y$-Cauchy sequences, and this defines $g(\hat{y})$
by continuity.
Therefore,
 the $G$-action on $(Y,d_Y)$ extends continuously to $\hat{Y}$.

\begin{proposition}\label{prop:liftedactiop} The group $G$ acts on $\hat{Y}$ as a geometrically finite convergence group with limit set $\pi^{-1}(\Lambda_G)$ and maximal parabolic subgroups $\PP_2$. 
\end{proposition}

\begin{proof}
Let us first check that we obtain a convergence action. We pick a sequence of distinct elements $(g_n)$. 
We may as well assume that there are two points
$a$ and $b$  in $\overline{Y}$ such that the sequence $(g_n)$ of homeomorphisms of the sphere tends uniformly to the constant map $a$ on the compact subsets of $S^2\setminus\{b\}$. 
When both $a$ and $b$ are distinct from the rank
1 parabolic points, then this property lifts to $\hat{Y}$. 

Let us assume that $a$ is a rank 1 parabolic point and write $\{x,x'\} =\pi^{-1}(a)$. Let us consider a regular neighborhood of type $(P)$. 
It defines two disjoint connected neighborhoods $W$ and $W'$ in $\hat{Y}$ of $x$ and $x'$ respectively. We may pick a point $z\in \hat{Y}$ and assume 
that $(g_n(z))$ tends to $x$ for instance, implying that $g_n(z)\in W$ for all $n$ large enough. 
Note that we may exhaust $\hat{Y}\setminus \pi^{-1}(\{b\})$ by connected compact subsets. Since $\pi^{-1}(a)$ is discrete, for any connected
compact subset $K\subset \overline{Y}\setminus \{b\}$ containing $\pi(z)$, the convergence property
implies that $g_n(K)$ has to 
be contained in $\pi(W)$ for $n$ large enough. This implies that $(g_n)$  tends to the constant $x$ in $\hat{Y}\setminus \pi^{-1}(b)$. If $b$ is not parabolic, then we are done. 

On the other hand, if $\pi^{-1}(b)=\{y,y'\}$, then the same reasoning for $(g_n^{-1})_n$ shows
that we may also assume that all compact subsets disjoint from $\{x,x'\}$ tend to $y$ under $(g_n^{-1})$. 
Let $V\subset\hat{Y}$ be a disk-neighborhood of $y$ disjoint from $W\cup W'$ 
and $K=\hat{Y}\setminus (W\cup W')$. Note that, for any $n$ large enough, $g_n^{-1}(K)$ is contained in $V$ so that the connected set $\hat{Y}\setminus V$ 
is covered by the two disjoint open sets $g_n^{-1}(W)$ and  $g_n^{-1}(W')$. The connectedness of $\hat{Y}\setminus V$ implies that
$g_n^{-1}(W')\subset V$ since $(g_n)$ pushes points into $W\subset (\hat{Y}\setminus V)$. Therefore, we have uniform convergence of $(g_n^{-1})$ on $W'$ to the constant map $y$. By symmetry, we get uniform convergence of $(g_n)$ to the constant map $x$ on compact subsets
disjoint from $y$. This shows that $G$ has also a convergence action on $\hat{Y}$.

Let us note that since the action of $G$ on  $\Lambda_G\cap Y$ is invariant and minimal, its closure $\hat{\Lambda}$ in $\hat{Y}$ will be a minimal invariant subset, hence
the limit set of this new action. 

We may check that the action on it is geometrically finite with maximal parabolic subgroups in $\PP_2$. 
\end{proof}

Since $\hat{Y}$ is planar, we may consider it as a subset of $S^2$ and  extend the action to the
whole sphere using \cite[Thm. 5.8]{ph:unifplanar}. There is only one such extention up to homeomorphism of the sphere since $\hat Y$ has no local cut points, see Corollary \ref{cor:tophatY}.  Since $\pi^{-1}(\Lambda_G)$ is a minimal invariant compact subset of the sphere, it remains the limit of this
extended action.

We will use the action of $G$ on $S^2$ to define the extension of $\pi$ to $S^2$. This will provide us with the cover of the action of $(G,\PP_2)$ over
$(G,\PP)$ in the sense of Martin and Skora.

\medskip

To conclude the proof of Theorem \ref{thm:blowup}, we establish the next proposition.

\begin{proposition}\label{prop:piext}
 
The projection $\pi:\hat{Y}\to \overline{Y}$ extends continuously to a degree-$1$ surjective $G$-equivariant map from the $2$-sphere to itself.    
\end{proposition}

\begin{proof} Let us consider a rank-one parabolic point $p$ in $S^2$ with parabolic subgroup $P\in\mathbb{P}_1$.

Let $H_P$ be the union of the two disjoint horoballs $H^+$ and $H^-$ attached to the parabolic point $p$ of $P$, as defined in Step 1.
There is an infinite order element $g\in G$ such that $T=\langle g\rangle$ has index at most $4$ in $P$. 

\medskip

Let us now consider the boundary Jordan curve $\gamma=\pi^{-1}(\partial H_P)\subset \hat{Y}$. It follows from Proposition~\ref{prop:tophatY} that $\gamma$ bounds a disk $D$ in the complement of $\hat{Y}$. The group $P$ acts on $\overline{D}$ as an elementary Fuchsian group with limit set $\{x,x'\}=\pi^{-1}(\{p\})$. The orbit space $(\overline{D}\setminus\{x,x'\})/P$ is a finite quotient of the compact annulus $(\overline{D}\setminus\{x,x'\})/T$. It follows that $P$ is a frieze group so that one may find a $P$-invariant line $\alpha$ in $D\cup\{x,x'\}$, joining $x$ and $x'$.

The curve $\alpha$ splits $\overline{D}$ into two closed half-disks $D^+$ and $D^{-}$ that we label so that $\pi(D^+\cap \partial D)= \partial H^+$ and $\pi(D^-\cap \partial D)= \partial H^-$.

The map $\pi$ defines a homeomorphism between the boundary of the quotient $\overline{D}\setminus(\alpha \cup \{x,x'\})/P$ and that of $H_P/P$. We now see that $\pi$ admits an equivariant surjective extension from $\overline{D}$ to $H_P\cup\{p\}$ that maps $\alpha$ to $p$ and
restricts to a homeomorphisms on the complement of $\alpha$. 

\medskip

We proceed in the same way for every parabolic subgroup in $\mathbb{P}_1$ and then extend the result $G$-equivariantly. As $\CCC$ is a null-sequence, one obtains in this way a continuous equivariant surjective extension $\pi:(S^2,\hat{Y})\to (S^2,\overline{Y})$. As the fibers satisfy the assumptions of Moore's theorem, it follows that $\pi$ is a degree-one map.
\end{proof}

\begin{proof}[Proof of Theorem \ref{thm:blowup}.] By Proposition \ref{prop:liftedactiop}, $G$ acts on $\hat{Y} \subset S^2$ as a geometrically finite convergence group, with limit set $\pi^{-1}(\Lambda_G)$, the pre-image of the limit set of the original action restricted to $Y$. In Proposition \ref{prop:piext} this action is extended to $S^2$ in a way that covers the original action of $G$ on $S^2$.  Since the action is properly discontinuous on the extended part, the action remains geometrically finite and and limit set remains the same. Proposition \ref{prop:piext} shows that the geometrically finite action of $(G, \mathcal{P})$ is covered by the action of $(G, \mathcal{P}_2)$ that we have constructed.   
\end{proof}
 
\section{Topological Schottky sets} 
\label{sec:topschott} 

\begin{definition} \label{def:schottky} 
  A {\it topological Schottky set} $\mathcal{S}$ is a proper compact subspace of $S^2$ defined by the following topological properties enjoyed by Schottky sets. 
\begin{itemize}
\item[(S1)] the set of components of $S^2\setminus \mathcal{S}$ is a countable non-empty collection of disks $\{D_i\}_{i\in I}$; 
\item[(S2)] for each $i$, $\bar D_i=D_i\cup\partial D_i$ is a closed disc; that is $D_i$ is a Jordan domain. 
\item[(S3)] for each pair $i\neq j\in I$ $\bar D_i$ and $\bar D_j$ meet in at most 
one point.
\item[(S4)] for each triple of distinct indices $i,j,k\in I$, $\bar D_i\cap \bar D_j \cap \bar D_k=\emptyset$, 
\item[(S5)] for every open cover ${\mathcal{U}}$ of $S^2$ and for all but finitely many $i\in I$ there is a $U_i\in{\mathcal{U}}$ such that $D_i\subset U_i$.
\end{itemize} 
\end{definition}

\begin{remark}
If the 2-sphere is endowed with a metric, the purely topological condition (S5) is equivalent to asking that 
$\mathcal{S}$ is an $E$-set (Def. \ref{def:nses}). This is an easy consequence of the Lebesgue number lemma. 
\end{remark}

The most well-known topological Schottky sets are the Sierpi\'nski carpet and the Apollonian Gasket.  These both occur as the limit sets of geometrically finite Kleinian groups, \cite{KK}, \cite{HPW}. Hence they are also the (Bowditch) boundaries of relatively hyperbolic groups. 
Observe that in contrast to the definition of a Schottky set, the cardinality of $I$ is not required to be at least $3$. However, if $|I|\le2$ then $\mathcal{S}$ {either violates (S3) or has non-empty interior.  Boundaries of relatively hyperbolic groups which are Schottky sets have empty interior by \ref{prop:ordcomp} }

\begin{proposition}\label{prop:schotoprop}
A topological Schottky set is connected, locally connected, hence arcwise connected, with no cut points and no cut pairs.

\end{proposition} 

\begin{lemma}\label{lma:semilocalcnd}
Let $\mathcal{S}$ be a topological Schottky set and $\Omega$ a non-empty open connected subset of $S^2$ such that, for each $i\in I$, $\partial D_i\cap \Omega$ is empty or connected.
The set $X = \mathcal{S}\cap \Omega$ is connected.
\end{lemma}
Note that the proof of this lemma does not use the $E$-set condition (S5). Moreover the conclusion of the lemma holds for any closed subset $S$ of the 2-sphere such that the intersection of the boundary of each complementary component with $\Omega$ is empty or connected. 
\begin{proof} 
 Let us consider two open subsets of $\Omega$, $R$ and $B$, for red and blue, such that $X\subset R\cup B$, $X\cap R$ 
 and $X\cap B$ are not empty, but $X\cap R \cap B=\emptyset$.
 We may assume that each component of $R$ and $B$ intersects $X$, by removing any components that do not intersect $X$ (note that the sphere is locally connected so every component of an open subset is itself open). 

We will increase these sets (by adding in disks associated to the two components) into two open and disjoint subsets that cover $\Omega$: this will prove that one of them has to be empty,
hence that $X$ is connected. With this in view, we split the set of components $\{D_i\}_{i\in I}$ into three sets  $I=I_0\sqcup I_R \sqcup I_B$.

Let $i\in I$, and let us write  $C_i= \partial D_i$. If $C_i\cap X=\emptyset$, since $\Omega$ is a connected set intersecting $\mathcal{S}$ and $C_i$ is a Jordan curve, then $\Omega\cap D_i=\emptyset$ and we let $i$ belong to $I_0$.   If not, $(C_i\cap X)$ is connected
by assumption, and covered by $R$ and $B$. Hence, $R\cap (C_i\cap X)=\emptyset$ or $B\cap (C_i\cap X)=\emptyset$. In the former
case, $D_i\cap R=\emptyset$ as each component of $R$ intersects $X$, but not $C_i$, so we let $i$ belong to $I_B$; in the latter, we let $i$ belong to $I_R$. Thus $i\in I_R$ if and only if $(C_i\cap X)\subset R$
and  $i\in I_B$ if and only if $(C_i\cap X)\subset B$. 

We let $$R'= R \cup (\cup_{i\in I_R} D_i \cap \Omega)\quad \hbox{and} \quad B'= B \cup (\cup_{i\in I_B} D_i\cap\Omega)\,.$$
We obtain in this way a cover of $\Omega$ by two disjoint open sets, so that one of them has to be empty. Therefore, one
of $R$ or $B$ has to be empty as well, establishing the connectedness of $X$. 
\end{proof}

\begin{proof}[Proof of Proposition \ref{prop:schotoprop}] To show that $\mathcal{S}$ is connected and with no cut points, we apply Lemma~\ref{lma:semilocalcnd} twice: with $\Omega =S^2$ first and then
with $\Omega=S^2\setminus \{x\}$, for any $x\in \mathcal{S}$. As each
 $C_i\subset \mathcal{S}$  is a closed simple curve, it cannot be disconnected by removing at most one point and it follows that $\mathcal{S}$ and $\mathcal{S}\setminus \{x\}$ are both connected.

Now let $x,y\in \mathcal{S}$ be two points and consider $\Omega=S^2\setminus \{x,y\}$. If no 
$C_i \subset \mathcal{S}$ contains both $x$ and $y$, the previous argument applies and we see that $x,y$ cannot form a cut pair. We can thus assume that there is an $i\in I$ such that $x,y\in C_i$. Let $\gamma$ be a properly embedded arc in $\Bar{D_i}$ connecting $x$ to $y$. $D_i\setminus\gamma$ is the union of two open disks, $D$ and $D'$, each adjacent to precisely one connected component of $C_i\setminus\{x,y\}$.  According to the remark before the proof of Lemma~\ref{lma:semilocalcnd} we can now apply Lemma~\ref{lma:semilocalcnd} to the closed set $S=\mathcal{S}\cup \gamma$ with $\Omega=S^2\setminus\gamma$  to conclude that $S\setminus\gamma=\mathcal{S}\setminus \{x,y\}$ must be connected. 
This shows that $\mathcal{S}$ has no cut pairs. One can also see this by noting that one circle cannot disconnect a Schottky set. Indeed, if one filled in the remaining circles the result is a disk. 
As $\mathcal{S}$ is an $E$-set, we deduce from \cite[Theorem VI.4.4]{whyburnantop} that it is also locally connected, hence arcwise connected \cite[Theorem II.5.1]{whyburnantop}.
\end{proof}

Our first key result is that the boundaries of the $D_i$ are topologically distinguished, generalizing the case of a Sierpi\'nski carpet.

\begin{proposition}
\label{boundarychar} 
 Let $\mathcal{S}$ be a topological Schottky set with $\mathcal{S} \simeq 
S^2 \setminus \cup(D_i)$, where each $D_i$ is open. Then the non-separating 
embedded circles of $\mathcal{S}$ are exactly the 
$C_i=\bar D_i \cap \mathcal{S}$. 
\end{proposition}

\begin{proof}
Let $C$ be an embedded circle in $\mathcal{S}\subset S^2$. By the Jordan curve
theorem, the complement of $C$ consists of two open discs $O$ and $O'$. 
Note that $C$ is a $C_i$ if and
only if either $O$ or $O'$ coincides with $D_i$. If this is not the case, both 
$O$ and $O'$ contain points of $\mathcal{S}$ and $C$ separates $\mathcal{S}$.

We want to show that if $C=C_i$ then $C$ does not separate $\mathcal{S}$.
Let us consider $\Omega = S^2\setminus \bar D_i$. Condition (S3) ensures that Lemma
\ref{lma:semilocalcnd} applies to prove the connectedness of $\mathcal{S}\setminus C_i$.
\end{proof}

\begin{corollary}\label{c:extension}
Any homeomorphism $h:\mathcal{S}_1\to \mathcal{S}_2$  between two topological  Schottky sets is the restriction of a
self-homeomorphism $H:S^2\to S^2$ of the sphere. \end{corollary}

This implies that we may define a topological Schottky set as an abstract compact subset homeomorphic to that of an embedded topological
Schottky set as above.

\begin{proof} By Proposition \ref{boundarychar}, $h$ maps boundary components $\{C_i^1\}$  to boundary components $\{C_i^2\}$. 
As these components are  Jordan curves, one may extend $h:C_i^1\to C_i^2$ as a homeomorphism $H_i:D_i^1\to D_i^2$ for 
each $i\in I$. Since topological Schottky sets are $E$-sets, these local homeomorphisms 
induce a global homeomorphism $H:S^2\to S^2$.
\end{proof}

\begin{proposition}\label{prop:schotplanar} Let $(G,\PP)$ be a relatively hyperbolic pair. If its Bowditch boundary is homeomorphic
to a topological Schottky set, then $(G,\PP)$ is a geometrically finite convergence group on $S^2$.
\end{proposition}

\begin{proof} We may assume that $G$ acts as a convergence group action on a topological Schottky $\mathcal{S}\subset S^2$. 
Proposition \ref{boundarychar} implies that $G$ preserves the collection of boundary circles. Therefore, we may apply
 \cite[Thm. 5.8]{ph:unifplanar} and extend in this way the action as a global convergence  group action on the sphere.
\end{proof}

\begin{corollary}\label{cor:parabschottky} Let $(G,\PP)$ be a relatively hyperbolic group pair with Bowditch boundary a topological Schottky set. 
The set $\cup_{i\neq j\in I}(\Bar{D_i}\cap \Bar{D_j})$ corresponds to the set of parabolic points whose stabilizers are 2-ended. In particular, if $(G, \PP)$ has a Schottky set boundary, then $G$ is finitely generated. 
\end{corollary}

\begin{proof} Let $p$ be a parabolic point. Let $\Omega_p$ denote the union of components of the ordinary set that contain $p$ on their
boundaries: according to the definition of a topological Schottky set, $\Omega_p$ is either empty, or has one or two components. By  Proposition \ref{prop:parabsubgp}, the action
on  $S^2\setminus (\{p\}\cup\Omega_p)$ is cocompact. 

If $\Omega_p=\emptyset$, then  $(S^2\setminus\{p\})/\Stab (p)$ is a compact surface orbifold.  If $\Omega_p$ has a single component,
then $\Stab (p)$ is cyclic since it preserves $\partial\Omega_p$, but this prevents the quotient to be compact on its complement as the action of 
the cyclic group is generated by a translation by Lemma \ref{lma:translation}. Therefore, if $\Omega_p$ is non-empty, then it is the
union of two disks. Conversely, if two boundary components intersect, then Proposition \ref{suswarup} implies that the intersection
of their stabilisers  is a parabolic point $p$. Up to index $2$, $\Stab (p)$ fixes each component, hence is a rank 1 parabolic point.

Let us check the group is finitely generated. It is enough to check that the peripheral subgroups are finitely generated, by \cite[Corollary p 371]{GPnonfg}. Stabilizers of parabolic points in the closure of the ordinary set are virtually cyclic. If there are any parabolic points that are not in the closure of the ordinary set, then by Proposition \ref{prop:parabsubgp}, the stabilizer is virtually a closed surface group and hence finitely generated.  
\end{proof}

\section{Incidence graphs for topological Schottky sets} 
\label{sec:incidence} 
We recall Definition \ref{def:schottky}.  A {\it topological Schottky set} $\mathcal{S}$  is a  connected, locally connected, 1-dimensional subset of the sphere such that
the complement is a union of pairwise disjoint Jordan domains. The closure of each component of the complement is homeomorphic to a disc $\Bar{D_i}$.  The intersection $\Bar{D_i} \cap \Bar D_j$,  $i\neq j$, is at most one point and a point of $\mathcal{S}$ belongs to at most two $(\Bar D_i)$. A topological Schottky set has no cut points nor cut pairs.

In this situation, we can draw more conclusions from the above construction in Section \ref{sec:blow}. 

\begin{definition} We define the \emph{incidence graph $\Gamma(\mathcal{S})$ of the topological Schottky set $\mathcal{S}$}.  
Let $\Gamma$ be the bipartite graph with vertex set the union of vertices $\{v_i\}_{i\in I}$, associated to the components 
$\{D_i\}_{i\in I}$ or, equivalently by Proposition~\ref{boundarychar}, to the embedded non-separating circles in $\mathcal{S}$, and vertices $v_p$, associated to intersections $\Bar{D_i}\cap \Bar{D_j}$, such that there is a non oriented edge between $v_i$ and $v_p$ if and only if $p\in \partial D_i$. 
Since we are working with a topological Schottky set, 
 $\Gamma$ injects continuously into $S^2$. 
To see this we pick for each component $D_i$ a base-point $v_i\in D_i$ and join $v_i$ to each $p\in \partial D_i \cap \partial D_j$ with an arc in $\Bar{D_i}$.

If $(G, \mathcal{P})$ is a relatively hyperbolic group pair whose boundary is a topological Schottky set, we will often denote this graph by $\Gamma(G)$ or $\Gamma(G, \mathcal{P})$. As observed in Lemma~\ref{lem:parabschottky}, each edge corresponds to a rank-1 parabolic point.
Also, we may ignore the vertices corresponding to the rank-1 parabolic points since this will not change the topology of the graph. 

\end{definition}

The following is a consequence of Proposition~\ref{boundarychar}.

\begin{lemma} 
If $\partial(G, \mathcal{P})$ has Bowditch boundary a topological Schottky set, then $G$ acts on $\Gamma(G)$. \end{lemma} 

The following was established in the proof of Corollary~\ref{cor:parabschottky}
    
\begin{lemma}\label{lem:parabschottky} 
Let $(G,\PP)$ be a relatively hyperbolic group pair with Bowditch boundary a topological Schottky
set.  The intersection of two $\partial D_i$, which corresponds to an edge in the incidence graph, is a parabolic point with a 2-ended stabiliser. All other parabolic points have stabilisers  
which are virtually
 compact surface orbifold groups. \end{lemma}

\begin{corollary}\label{cor:comp}
Let $(G,\PP)$ be a relatively hyperbolic group pair with Bowditch boundary a topological Schottky
set. Let $\mathcal{P} = \PP_1 \cup \PP_2$ be as in Definition \ref{def:parabolicsplit} and consider the geometrically finite action of $(G,\PP_2)$ provided by
Theorem \ref{thm:blowup}.   The components of the ordinary set of $\partial(G,\PP_2)$ are in bijection with the components
of the incidence graph $\Gamma$. There is a continuous injective map of $\Gamma$ into $S^2$ such that the image of each cycle in $\Gamma$ separates the Bowditch boundary of $(G,\PP_2)$. \end{corollary}

\begin{proof} We will use the same notation introduced in Section~\ref{sec:blow} for the proof of Theorem~\ref{thm:blowup}: $Y$ is the complement of the union of pairs $K$ of closed horoballs attached to each rank-1 parabolic point $p$ and $\hat{Y}$ its completion after blowing-up, so that there is a natural quotient map $\pi:\hat{Y}\longrightarrow \Bar{Y}$.
Let us consider $\Gamma$ as a subset of $S^2$, cf. Def. 6.1, and let $\Gamma_T=\Gamma \cap Y$ and let $\Gamma_T'$ be its closure in $\hat{Y}$.
This graph is 
disjoint from the limit set, and each edge is cut into two pieces by  a Jordan domain  $D$ bounded by the circle $\pi^{-1}(K)$, $K\in\CCC$. We may then
connect both sides of the edge of $\Gamma'_T$ in $D$ to reconstruct a graph $\Gamma'$ isomorphic to $\Gamma$ which will now be disjoint from the limit set. 

Let us observe that this edge separates in $\overline{D}$ the preimages of the parabolic point, so that any cycle in $\Gamma'$ separates the limit set. 
By construction each connected component of the new ordinary set contains a component of $\Gamma'$ (which might be reduced to a single point). 
To see there is at most one,  we may proceed by contradiction as follows: if two components of $\Gamma'$ belonged to the same component of $\hat{\Omega}_G$,
we could consider a curve joining them in $Y$: a contradiction.
\end{proof}

\begin{thmincidence} Let $\mathcal{S}$ be a topological Schottky set with $\mathcal{S} = \partial(G, \mathcal{P})$.    Then the incidence graph $\Gamma(\mathcal{S})$ has 1, 2 or infinitely many components. The stabilizer of each component is virtually the deck transformation group of a regular planar covering of a closed surface.
\end{thmincidence} 

\begin{proof}  Let $\mathcal{S} \simeq 
S^2 \setminus \cup(D_i)$ and $C_i=\partial\Bar{D_i}\subset \mathcal{S}$. 
According to Proposition~\ref{boundarychar}, 
the circles $C_i$ are precisely those that do not separate $\mathcal{S}$ so the group $G$ must map their collection to itself (see also Corollary~\ref{c:extension}).
Therefore,
\cite[Thm. 5.8]{ph:unifplanar} enables us to extend the action onto the whole sphere. The parabolic points are either surface groups that are not accessible from any  component
or rank 1 parabolic points, which correspond to two intersecting disks, as observed in Lemma~\ref{lem:parabschottky}. 
By Lemma \ref{lma:omegafuchs}, there are finitely many orbits of components of $\Omega_G \,(= S^2\setminus \mathcal{S})$ under $G$, and the stabilizer of each one of them is isomorphic
to a geometrically finite Fuchsian group. This enables us to equip $\Omega_G$ with the Poincar\'e metric on each of its components, so that the action of $G$
is distance preserving, and to consider
a fundamental domain $\cD\subset \Omega_G$ given by a finite union of finite sided polygons.

According to Corollary~\ref{cor:comp}, the components of the graph are thus in bijection with those of the 
ordinary set $\Omega_2$ of $(G,\PP_2$) obtained by blowing up the rank 1 elements of $\PP$. Since the action is geometrically finite, there
are 1, 2 or infinitely many components of $\Omega_2$ as seen in Proposition~\ref{prop:ordcomp}.  Note that the components of the limit set of $(G,\PP_2)$ 
form an $E$-set according to Proposition \ref{prop:infends}. Therefore this is also the case for the components of $\Omega_2$.
By Proposition \ref{prop:nullseq} applied to the collection of boundaries of components of $\Omega_2$, 
we deduce that the stabilizers of  components of $\Omega_2$ are relatively quasiconvex subgroups.  

Let $\cD_2=\phi^{-1}(\cD)\subset \Omega_2$,  where $\phi:S^2\to S^2$ denotes  the degree one quotient map provided by Theorem \ref{thm:blowup}. The fundamental domain $\cD$ has only finitely many
ideal vertices, which correspond to rank 1 parabolic points. 
The two edges incident to a parabolic point are mapped to each other by a parabolic element.
It follows that $\overline{\cD_2}$ is a compact subset of $\Omega_2$, and a fundamental domain
for the action of $G$ on $\Omega_2$. Therefore $\Omega_2/G$ is a finite union of closed $2$-orbifolds.

We observe now that the stabilizer of each component $O$ of $\Omega_2$ is precisely the group of deck transformations of the regular planar covering $O\to \Sigma\subset\Omega_2/G$ 
which is the quotient of a closed $2$-orbifold group $\pi_1(\Sigma)$ by its normal subgroup $\pi_1(O)$. Whenever $O$ is simply connected this is in fact a genuine closed 
$2$-orbifold group. As the stabilizer of a component of the incident graph coincides with the stabilizer of the component $O$ in which it is contained, the conclusion follows.
\end{proof}

\section{One component in the incidence graph}
\label{sec:one} 

Here we prove Theorem \ref{thm:one}: 

\begin{thmone}
Let $\mathcal{S}$ be a topological Schottky set with $\mathcal{S} = \partial(G, \mathcal{P})$.
 
When the incidence graph $\Gamma(\mathcal{S})$ has one component, then $G$ is virtually a free product of a free group $F_n$ of rank $n\geq 0$ and some finite index subgroups of groups in $\mathcal{P}$. Moreover, its action is faithful and orientation preserving, and $G$ is covered by a geometrically finite Kleinian group $K$.
\end{thmone}

Recall from Theorem \ref{thm:blowup}  and Definition \ref{def:parabolicsplit} that if $(G, \mathcal{P})$ is a relatively hyperbolic group pair and  $\mathcal{P}_2$ is the set of non 2-ended subgroups of $\mathcal{P}$, then $(G, \mathcal{P}_2)$ is a relatively hyperbolic group pair and the Bowditch boundary $\partial(G, \mathcal{P}_2)$ is obtained from $\partial(G, \mathcal{P})$ by unpinching the 
parabolic points of $\partial(G, \mathcal{P})$
with two-ended stabilizers. Furthermore in our situation the unpinched boundary $\partial(G, \mathcal{P}_2)$ is also planar. By Corollary \ref{cor:parabschottky}, $G$ is finitely generated.

There are three cases to consider for relatively hyperbolic group pairs with Schottky set boundary.  The first is when the incidence graph has one component. 

\begin{theorem} \label{thm:cantorset} Let $(G, \mathcal{P})$ be a relatively hyperbolic group pair such that $ \partial(G, \mathcal{P})$ is a Schottky set with connected incidence graph.  Let  $\partial (G, \mathcal{P}_2)$ be the relatively hyperbolic group pair where  $\mathcal{P}_2$ consists of the subgroups in $\mathcal{P}$ that are not two-ended. 
Then  $\partial (G, \mathcal{P}_2)$
 is a Cantor set. \end{theorem} 

\begin{proof} 
We will prove that if  $\partial (G, \mathcal{P}_2)$ has a non-trivial component, it is a dendrite. However, this is impossible according to Lemma~\ref{cor:nodendrite}. The theorem will then follow.

Take a component $L$ of  $\partial (G, \mathcal{P}_2)$.  Suppose that $L$ contains at least two points $x$ and $y$.

\begin{itemize} 
\item $L$ is a connected, locally connected compact metrizable space. The component is connected by definition.  
A component $L$ is itself the boundary of a relatively hyperbolic subgroup pair: the subgroup stabilizing $L$ along with the peripheral subgroups whose fixed points belong to $L$ \cite{bowrelhyp}.   Thus $L$ is compact and a metric space.  Furthermore, the set of peripheral subgroups is a subset of $\mathcal{P}_2$, each of whose elements is virtually a closed surface group, by Proposition \ref{prop:parabsubgp}.  Therefore by Bowditch \cite{bowconnect} the boundary $L$ is locally connected. 

\item The component $L$ contains no simple closed curve. Any simple closed curve  bounds two discs in $S^2$ which are either contained in $L$ or not.  At least one must be contained in $L$ as the complementary region would correspond to an additional component of the incidence graph, 
 by Corollary \ref{cor:comp}. This contradicts our assumption that the incidence graph is connected.
If the simple closed curve bounds a disk in $L$, then the boundary has non-empty interior
{which contradicts the last part of Proposition \ref{prop:ordcomp}. }
\end{itemize}

Thus, $L$ should be a dendrite, so we may now conclude that there are no non-trivial components. 
\end{proof} 

\begin{proof}[Proof of Theorem \ref{thm:one}] 
 $\partial (G, \mathcal{P}_2)$ be the relatively hyperbolic group pair where  $\mathcal{P}_2$ consists of the subgroups in $\mathcal{P}$ that are not two-ended.  According to Theorem \ref{thm:cantorset},  the Bowditch boundary of the group pair $(G, \mathcal{P}_2)$ is a Cantor set, hence its ordinary set  on $S^2$ is connected.
 We are now in a position to apply Theorem~\ref{thm:Cantor} and conclude that, in this case, the group $G$ is virtually a free product of infinite cyclic groups and finite index subgroups of peripheral groups, which are virtual surface groups, recall Proposition \ref{prop:parabsubgp} and Corollary \ref{cor:parabschottky}. It follows from Corollary \ref{cor:1or2}  that $G$ is covered by a Kleinian group if it is finitely generated. By Corollary \ref{cor:parabschottky}, $G$ is finitely generated, and the conclusion follows.
\end{proof} 

\begin{remark}
In the same circle of ideas, Otal proves that if $(\mathbb{F},\PP)$ is a free relatively hyperbolic group pair such that its Bowditch boundary is a topological Schottky set, then there exists a handlebody with fundamental group $\mathbb{F}$ and disjoint homotopy classes of simple curves on its boundary that represent the peripheral structure $\PP$ \cite{Otal}.
\end{remark}
\section{More components in the incidence graph} \label{sec:more}

In the previous section, under the hypothesis that $\partial(G, \mathcal{P})$ is a topological Schottky set with connected incidence graph, we determined the structure of the group $G$.
Since the incidence graph has 1, 2, or infinitely many components, we now analyze what happens in the latter two cases. 

\begin{thmtwo}  Let $\mathcal{S}$ be a topological Schottky set with $\mathcal{S} = \partial(G, \mathcal{P})$.
When the incidence graph $\Gamma(\mathcal{S})$  has exactly  2 components $G$ is virtually a closed surface group. \end{thmtwo}

\begin{proof} 
We recall that the rank 1 parabolic points in $\mathcal{P}$ correspond to the edges of the incidence graph by Lemma~\ref{lem:parabschottky} and the very definition of the incidence graph.

Then, we unpinch the rank-1 parabolic points as in Theorem \ref{thm:blowup}.  This results in a different geometrically finite action of the group $G$. For every parabolic point removed, the two components of the domain of discontinuity that corresponded to the endpoints of the edge are contained in the same component by  Corollary \ref{cor:comp}.  So when there are no more rank-1 parabolic points, there are two components of the domain of discontinuity. Then by Corollary \ref{cor:1or2}, $G$ is virtually Fuchsian with limit set $S^1$. 

Since we already removed all of the rank-1 parabolic points, 
$G$ is virtually a closed surface group and $P_2 = \emptyset$. 
\end{proof} 

When the incidence graph has infinitely many components, the topology of the blown-up limit set can be extremely varied so there is no hope of getting a meaningful description of the underlying group. 
Indeed, the next theorem shows in particular that the limit set of any finitely generated Kleinian group with infinitely many components in its 
domain of discontinuity and no two-ended parabolic subgroups is isomorphic to the boundary of some $(G,\mathcal P_2)$ obtained by blowing up all the rank-one parabolics of a relatively hyperbolic group $(G,\mathcal P_1\cup \mathcal P_2)$ where $\partial(G,\mathcal P_1\cup \mathcal P_2)$ is a topological Schottky set. 

\begin{thminfty} 
Let $K$ be a geometrically finite Kleinian group with non-empty domain of discontinuity. Then there is a peripheral structure  $\mathcal{P}_{K'}$ on a finite index subgroup $K'$ of $K$, such that $(K', \mathcal{P}_{K'})$  is a relatively hyperbolic group pair and $\partial(K', \mathcal{P}_{K'})$ is a topological Schottky set. 
 Moreover, $\mathcal{P}_{K'}$ contains 
 the natural peripheral structure of the Kleinian group $K'\subset K$.
\end{thminfty}

\begin{figure}[!h]
\begin{center}
\includegraphics[scale=.8]{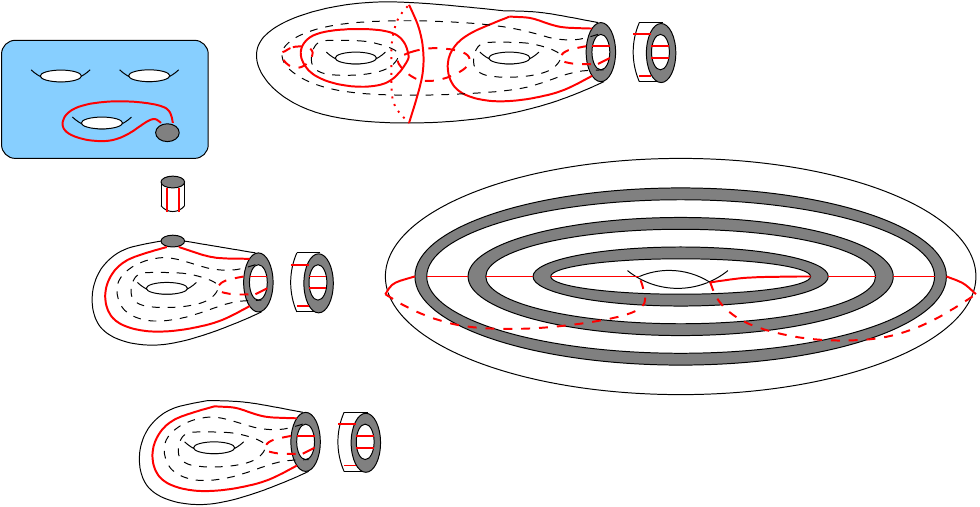}
\caption{An illustration of Theorem~\ref{thm:infty}. The picture represents a manifold $M_{K'}$ and the choice of a peripheral structure that makes it hyperbolic with totally geodesic boundary. A compressing disk separates the manifold into a hyperbolic piece with totally geodesic boundary of genus 3 (represented schematically by the light blue piece at the top left) and another part that coincides with its characteristic sub-manifold. This consists of three one-holed surfaces-times-interval (one of genus 2 and two of genus 1) that are attached to a solid torus along annuli. Thickenings of the compression disk and JSJ annuli are also shown. A disk-and-annuli busting paring is in red. Paring curves on ``interior surfaces" of the three products are dashed, as are the interior surfaces themselves. Different pieces are glued along shaded disks and annuli.}
\end{center}
\end{figure}

\begin{proof} 
We choose $K'$ to be a torsion-free finite-index subgroup of $K$ contained in $PSL(2, \mathbb{C})$.
Below we will define a peripheral structure $\mathcal{P}_{K'}$ of $K'$ that will contain all the parabolic subgroups of $K' < PSL(2, \mathbb{C})$ but will in general be larger. 

In this situation, there is 
an irreducible and orientable manifold with boundary $M_{K'}$ obtained as the quotient of the 1-neighborhood of the convex hull of $\Lambda_{K'}$, the limit set of $K'$, by the action of $K'$.  
There is at least one geometrically finite end, as the group is geometrically finite and its limit set is not all of $S^2$. This manifold comes equipped with a natural pared structure, given by the parabolic structure on $K'$. This realizes the boundary of $M_{K'}$ as a union of connected surfaces with boundary, which corresponds to the rank-1 cusps in the hyperbolic structure.  We will add curves to the peripheral structure so that the resulting pared manifold contains no essential annuli or disks, and thus admits a hyperbolic structure with totally geodesic boundary \cite[Theorem B' page 70]{Morgan}.

We will first consider the case when these surfaces are incompressible. Now, in this situation, $M_{K'}$ admits a JSJ-decomposition along a finite family of pairwise disjoint and non parallel incompressible annuli $A_i$ into ``geometric pieces'' (see \cite{characteristic} for a description): $I$-bundles over surfaces (Seifert fibered pieces) and anannular manifolds with boundary (hyperbolic pieces). By taking a further cover if necessary, that is by taking a further finite-index subgroup, we assume no twisted $I$-bundle appears in the decomposition. Note that a piece can have different structures. For instance, a solid torus can be seen as a circle bundle over a disk, an interval times an annulus, as well as a twisted $I$-bundle over a M\"{o}bius band. We only require each piece to admit some product structure.

The characteristic submanifold $C_{K'}$ in $M_{K'}$ consists of all the surface-times-interval components together with small neighborhoods of the JSJ annuli $A_i$, which are solid tori $T_i$. Note that if $C_{K'}$ is empty, $M_{K'}$ with its natural pared structure admits a hyperbolic metric with totally geodesic boundary perhaps with rank-one cusps so that $\partial(K', \mathcal{P}_{K'})$ is already a topological Schottky set and notably a Sierpi\'nski carpet if $\mathcal{P}_{K'}$ is empty.

Otherwise, we observe that the boundary of each solid torus $T_i$ is partitioned into four annuli: two of them contained in $\partial M_{K'}$ and two others properly embedded in $M_{K'}$ and parallel to $A_i$. For each $T_i$, we mark two points on each of the four circles that delimit the four annuli in its boundary. We then connect these pairs of points with two arcs in  each annulus of $\partial T_i\cap \partial M_{K'}$ running from one circle to the other.

Remark that $\partial C_{K'}\setminus \partial M_{K'}$ consists of properly embedded annuli contained in the boundary of some tori $T_i$. The rest of the boundary $\partial C_{K'}$ in $\partial M_{K'}$ is a union of subsurfaces, possibly with boundary or cusps.  

For each complementary piece of $\partial M_{K'} \setminus \partial C_{K'}$, we connect all the marked points on its boundary components with an embedded collection of essential (pairwise non-parallel) arcs. Next, if some component of $\partial M_{K'}\setminus \cup_i T_i$ is an annulus, (for instance, if a piece is a solid torus) we connect the pair of points on one boundary component directly with the pair of points on the other boundary component. 
Each remaining component of $C_{K'}\setminus \cup_i T_i$ is a surface times an interval $S \times I$. In this case again we first connect the marked points on the boundary circles along an embedded collection of essential arcs in $ S \times I \cap \partial M_{K'}$ (as was done in the complementary components). Then we take a pair of pants decomposition of each remaining component after cutting along these arcs. The pair of pants decomposition for the pieces of $S \times \lbrace 0 \rbrace$ should be different from the decomposition for $S \times \lbrace 1 \rbrace$, in particular, the curves of the pants decomposition for $S \times \lbrace 0 \rbrace$ should be transverse to curves going through $S \times \lbrace 1 \rbrace$. Since there are two arcs meeting at each marked point, the union of these arcs and curves is a collection of curves so that any essential annulus in $\partial C_{K'}$ is transverse to some curve in this collection.  Since these curves are essential and non-parallel, we can make this collection peripheral.  The resulting pared manifold with this peripheral structure will admit a Kleinian representation where the quotient of the $1$-neighborhood of the convex hull of the limit set is a hyperbolic manifold with totally geodesic boundary.  Therefore its limit set can be realized as a Schottky set.

Assume now that $\partial M_{K'}$ is compressible. In this case, the limit set of $K'$ is not connected. We can choose a finite family $\mathcal{D}$ of properly embedded pairwise disjoint essential disks such that (the closure of) each component of the complement of the disks, $M_{K'}\setminus\cup_{D\in{\mathcal{D}}}D$, has incompressible boundary and the family ${\mathcal{D}}$ is minimal with respect to this property. 
As we did with the JSJ-annuli in the previous case, for each disk $D$ we remove small cylindrical neighborhood $C_D$ and mark two points on each of the circles delimiting the two disks on the boundary of $C_D$. We then connect the two pairs of points by two arcs in the annulus contained in $\partial C_D$.
For each component $N$ of $M_{K'}\setminus\cup_{C_D\in{\mathcal{D}}}D$ let us denote $C_N$ the characteristic submanifold of $N$. Note that we can assume that the annuli of the JSJ-decomposition of $N$ are disjoint from the disks of the family ${\mathcal{D}}$. We can now repeat the previous argument keeping in mind that this time we need to connect also the marked points on the boundary of the disks.  
 \end{proof}


\begin{thebibliography}{McM2}

\bibitem{BKM} M. Bonk, B. Kleiner and S. Merenkov, {\em Rigidity of Schottky sets,} Amer. Jour. Math. {\bf 131} (2009) 409--443.   



\bibitem{bowbound} B. H. Bowditch, {\it Boundaries of geometrically finite groups,} Math.  Zeit. {\bf 230} (1999) 509--527. 

\bibitem{bowconnect} B. H.  Bowditch, {\it Connectedness properties of limit sets,} Trans. Amer. Math. Soc. {\bf 351} (1999), no. 9, 3673--3686. 

\bibitem{bowrelhyp} B. H. Bowditch, {\em Relatively hyperbolic groups} Internat. J. Algebra and Computation. {\bf 22} (2012) 66 pp. 


\bibitem{bridson:haefliger}
M.~R. Bridson and A. Haefliger.
\newblock {\em Metric spaces of non-positive curvature}, volume 319 of {\em
  Grundlehren der Mathematischen Wissenschaften [Fundamental Principles of
  Mathematical Sciences]}.
\newblock Springer-Verlag, Berlin (1999).

\bibitem{BuyaloS} 
S. Buyalo and V. Schroeder. 
\newblock {\em Elements of Asymptotic Geometry}, 
EMS Monographs in Mathematics, European Mathematical Society. (2007). 

\bibitem{Brock} J. F. Brock, {\it Iteration of mapping classes on a Bers slice: examples of algebraic and geometric limits of hyperbolic 3-manifolds,} In Lipa's Legacy, J. Dodziuk, L. Keen, ed., Proceedings of the Bers Colloquium, 1997, pp. 81--106.

\bibitem{claytor}
S. Claytor,
\newblock {Topological immersion of {P}eanian continua in a spherical surface},
\newblock {\em Ann. of Math. (2)} {\bf 35} (1934) 809--835.

\bibitem{Cannon91} J. W. Cannon,
{\it The theory of negatively curved spaces and groups}, in {\it Ergodic theory, symbolic dynamics, and hyperbolic spaces}, Trieste, 1989,
T. Bedford and M. Keane and C. Series, ed. 
Oxford Sci. Publ. (1991) 315--369. 
 


\bibitem{Dasgupta:Hruska} A. Dasgupta and G. C. Hruska, {\it Local connectedness of boundaries for relatively hyperbolic groups}, J. Topol., {\bf 17}:e12347 (2024)

\bibitem{daverman:decompositions}
R.~J. Daverman.
\newblock {\em Decompositions of manifolds}, volume 124 of {\em Pure and
  Applied Mathematics}.
\newblock Academic Press Inc., Orlando, FL, 1986.



\bibitem{GM1} F. W. Gehring and G. J. Martin, {\it Discrete quasiconformal groups I}, Proc. Lond. Math. Soc. {\bf 55} (1987) 331--358.

\bibitem{GPnonfg} V. Gerasimov and L. Potyagailo, {\it Non-finitely generated relatively hyperbolic groups and Floyd quasiconvexity}, Groups, Geom. and Dyn. {\bf 9} (2015) 369--434.  

 \bibitem{GPdyn}V. Gerasimov and L. Potyagailo, {\it Quasiconvexity in relatively hyperbolic
groups}, J. Reine Angew. Math. {\bf 710} (2016), 95--135. 



\bibitem{ghys:delaharpe:groupes}
{\'E}. Ghys and P. de~la Harpe, editors.
\newblock {\em Sur les groupes hyperboliques d'apr\`es {M}ikhael {G}romov},
  volume~83 of {\em Progress in Mathematics}.
\newblock Birkh\"auser Boston Inc., Boston, MA (1990).
\newblock Papers from the Swiss Seminar on Hyperbolic Groups held in Bern,
  1988.

\bibitem{GMRS} R. Gitik, Mahan, E. Rips, and M. Sageev, {\it Widths of Subgroups}.  Trans. Am. Math. Soc. {\bf 350} (1998), no. 1, 321--329. 

\bibitem{ph:unifplanar}
P. Ha{\"{\i}}ssinsky,
\newblock {Hyperbolic groups with planar boundaries},
\newblock {\em Invent. Math.} {\bf 201} (2015), no. 1, 239--307. 

\bibitem{peter:cyril}
P. Ha{\"{\i}}ssinsky and C. Lecuire,
{\it Quasi-isometric rigidity of 3-manifold groups}, 
\newblock Preprint 2020, arXiv:2005.06813. 


\bibitem{HPW}  P. Ha\"issinsky,  L. Paoluzzi, and G. S. Walsh, {\it Boundaries of Kleinian groups}. Ill.  Jour. Math.s (Special Haken Issue)  {\bf 60} (2016), no. 1, 353--364. 

\bibitem{Hruskasurvey} G. C. Hruska, {\it Relative hyperbolicity and relative quasiconvexity for countable groups}. Alg. Geom. Top. {\bf 10} (2010) 1807--1856. 

\bibitem{hruska:walsh} G. C. Hruska and G. S. Walsh, {\it Planar boundaries and parabolic subgroups}, Math. Res. Lett. {\bf 30} no. 4 (2023) 1081--1112. 

\bibitem{KK} M. Kapovich and B. Kleiner, {\it Hyperbolic groups with low-dimensional boundary}, Ann. Scient. \'Ec. Norm. Sup. $4^e$ s\'erie {\bf 33} (2000) 647--669.


\bibitem{martin:skora}
G.~J. Martin and R.~K. Skora,
\newblock {Group actions of the {$2$}-sphere},
\newblock {\em Amer. J. Math.} {\bf 111} (1989) 387--402.


\bibitem{martin:tukia:invpairs}
G. J. Martin and P. Tukia,
\newblock{Convergence groups with an invariant component pair},
\newblock {\em Amer. J. Math.} {\bf 114} (1992), 1049--1077.

\bibitem{moore}
R. L. Moore, {\it Concerning upper-semicontinuous collections of continua}, {\em Trans.
A.M.S.} {\bf 27} (1925) 416--428.

\bibitem{Morgan} J. W. Morgan and H. Bass, Ed.  {\it The Smith Conjecture} Pure and Applied Mathematics, Academic Press,  1984. 

\bibitem{Osin06} D. V. Osin, {\it Relatively hyperbolic groups: Intrinsic geometry, algebraic properties, and algorithmic problems}.  Memoirs AMS 179 (2006), no. 843. 

\bibitem{Otal}
J.-P. Otal.
\newblock {\em Certaines relations d'\'equivalence sur l'ensemble des bouts d'un
  groupe libre}.
\newblock {J. London Math. Soc. (2)}  {\bf 46} (1992), 123--139.


\bibitem{scott:wall}
P. Scott and T. Wall.
\newblock {Topological methods in group theory}.
\newblock In {\em Homological group theory ({P}roc. {S}ympos., {D}urham,
  1977)}, volume~36 of {\em London Math. Soc. Lecture Note Ser.}, pages
  137--203. Cambridge Univ. Press, Cambridge, 1979.


\bibitem{hung} H. C. Tran, {\em Relations between various boundaries of relatively hyperbolic groups}, Internat. J. Algebra. Comp. {\bf 23} (2013), no. 7, 1551--1572. 

 \bibitem{tukia:convergence_groups}
P. Tukia.
\newblock {Convergence groups and {G}romov's metric hyperbolic spaces}.
\newblock {\em New Zealand J. Math.} {\bf 23}(1994), 157--187.


\bibitem{TukiaConical} P. Tukia, Conical limit points and uniform convergence groups, {\it J. reine. angew. Math.} {\bf 501} (1998) 71--98. 


\bibitem{characteristic} G. S. Walsh, The bumping set and the characteristic manifold. {
\em Algebraic and Geometric Topology} {\bf 14} (2014)  283--297.  

\bibitem{whyburnantop} G. T. Whyburn, {Analytic Topology}. American Mathematical Society, 1942. 


\bibitem{Yang} W.-Y. Yang, {\it Limit sets of relatively hyperbolic groups}, Geom. Dedicata {\bf 156} (2012), 1--12. 

\end{thebibliography}
\end{document}